\documentclass[preprint, 12pt, sort&compress]{elsarticle}

\usepackage{latexsym}
\usepackage{amsfonts}
\usepackage{graphicx,amssymb,natbib}
\usepackage{amsmath}
\usepackage{amssymb}
\usepackage[mathscr]{eucal}
\usepackage{enumerate}
\usepackage{amsthm}

\usepackage{bm}

\usepackage{color, soul}

\begin{document}

\newtheorem{tm}{Theorem}[section]
\newtheorem{pp}[tm]{Proposition}
\newtheorem{lm}[tm]{Lemma}
\newtheorem{df}[tm]{Definition}
\newtheorem{tl}[tm]{Corollary}
\newtheorem{re}[tm]{Remark}
\newtheorem{eap}[tm]{Example}

\newcommand{\pof}{\noindent {\bf Proof} }
\newcommand{\ep}{$\quad \Box$}

\newcommand{\al}{\alpha}
\newcommand{\be}{\beta}
\newcommand{\var}{\varepsilon}
\newcommand{\la}{\lambda}
\newcommand{\de}{\delta}
\newcommand{\str}{\stackrel}

\renewcommand{\proofname}{\bf Proof}

\allowdisplaybreaks

\begin{frontmatter}

\title{Relations of endograph metric and $\Gamma$-convergence on fuzzy sets
\tnoteref{usc}
 }
\tnotetext[usc]{Project supported by
 Natural Science Foundation of Fujian Province of China (No. 2020J01706)}
\author{Huan Huang}
 \ead{hhuangjy@126.com }
\address{Department of Mathematics, Jimei
University, Xiamen 361021, China}

\date{}

\begin{abstract}

This paper discusses the compatibility of
 the endograph metric and the $\Gamma$-convergence
on fuzzy sets in $\mathbb{R}^m$.
We significantly improve
the corresponding conclusion given in our previous paper
[H. Huang, Characterizations of endograph metric and $\Gamma$-convergence on fuzzy sets, Fuzzy Sets and Systems 350 (2018),  55-84] by relaxing its prerequisite.
Compared with the previous result,
the conclusion in this paper
shows that
not only
 the condition of the compactness of the positive $\al$-cuts
of the convergent fuzzy set sequence
can be deleted,
but also
 the condition of the connectedness
of
  all the positive $\al$-cuts
of the convergent fuzzy set sequence
can be replaced by a much weaker condition introduced in this paper,
which is called the connectedness condition.
At last, we discuss the properties
of the fuzzy set sequences with the connectedness condition
as a supplement.

\end{abstract}


\end{frontmatter}

\date{}

\section{Introduction}

Fuzzy theory and applications attracts much attention \cite{da, du, wu, can, kupka, garcia, wa, gong}.
Fuzzy set is a fundamental tool to investigate fuzzy phenomenon.
The endograph metric $H_{\rm end}$ on fuzzy sets
 was introduced by
Kloeden \cite{kloeden2}.
The $\Gamma$-convergence on fuzzy sets was
introduced by Rojas-Medar and Rom\'{a}n-Flores \cite{rojas}.
In recent years, the endograph metric $H_{\rm end}$
and
the $\Gamma$-convergence have received deserving attention
\cite{huang7, can, kloeden, kloeden2, kupka, rojas, huang}.
In general, the $\Gamma$-convergence is weaker than the Hausdorff metric convergence.

The work of Wang and Wei \cite{wan} indicated that the $\Gamma$-convergence on upper semi-continuous
fuzzy sets in $\mathbb{R}^m$, the $m$-dimensional Euclidean space, is metrizable.
In \cite{huang}, we found
that
the endograph metric $H_{\rm end}$ and the $\Gamma$-convergence
 is compatible on upper semi-continuous fuzzy sets in $\mathbb{R}^m$
under certain prerequisite.
In more detail, it is shown
that
the endograph metric $H_{\rm end}$ and the $\Gamma$-convergence
 is compatible under
the condition that the positive $\al$-cuts
of the convergent fuzzy set sequence is compact and connected,
and
the positive $\al$-cuts
of limit fuzzy set is compact.

In this paper, we improve the conclusion in \cite{huang} by relaxing its prerequisite.

First, we give a conclusion which says that, in normal fuzzy set case, the conclusion in \cite{huang}
is still true
with the condition of the compactness of the positive $\al$-cuts
of the convergent fuzzy set sequence being deleted.

More importantly, adopting a new proof idea, we
further give a more general conclusion on compatibility of the endograph metric $H_{\rm end}$ and the $\Gamma$-convergence.
 This conclusion
greatly improves the result obtained in \cite{huang}.

This conclusion applies to the general fuzzy sets
which are allowed to not satisfy the normality.
The first conclusion in normal fuzzy set case is an immediate corollary of this conclusion.

This conclusion
shows that
not only
 the condition of the compactness of the positive $\al$-cuts
of the convergent sequence
is not needed,
but also
 the condition of the connectedness
of
  all the positive $\al$-cuts
of the convergent sequence
can be replaced by a much weaker condition.
We call this condition the connectedness condition.

We point out that there exist a large variety
of convergent fuzzy set sequences and their limit fuzzy sets in $\mathbb{R}^m$,
which
 satisfy
the prerequisite of
the above general conclusion, but not satisfy the prerequisite
of the corresponding conclusion in \cite{huang}.
Examples are given to illustrated this fact.
This fact indicates that the conclusion using the connectedness condition in this paper
significantly improves the corresponding result in \cite{huang}.

At last, we conduct some supplementary discussions on the fuzzy set sequences with the connectedness condition.

The remainder of this paper is organized as follows.
In Section 2, we recall and give some basic notions and fundamental results related to fuzzy sets
and
the endograph metric and the $\Gamma$-convergence on them.
In Section 3, we give representation theorems for various kinds of fuzzy sets which are useful in this paper.
In Section \ref{nce},
we give a conclusion on the compatibility of
the endograph metric $H_{\rm end}$ and the $\Gamma$-convergence,
which is
an improvement of the corresponding conclusion in \cite{huang} in a special case.
In Section \ref{nceu}, we
give a general conclusion on the compatibility of the endograph metric $H_{\rm end}$ and the $\Gamma$-convergence.
 This conclusion
greatly improves both the corresponding result in \cite{huang}
and the result in Section \ref{nce}.
In Section 6, we conduct some supplementary discussions
on
the fuzzy set sequences with the connectedness condition.
By last, we draw conclusion
in Section 7.

\section{Fuzzy sets and endograph metric and $\Gamma$-convergence on them}

In this section, we recall and give some basic notions and fundamental results related to fuzzy sets
and
the endograph metric and the $\Gamma$-convergence on them.
Readers
can refer to \cite{wu, da, du, rojas} for related contents.

Let $\mathbb{N}$ denote the set of natural numbers.
Let $\mathbb{R} $ denote the set of real numbers.
Let $\mathbb{R}^m$, $m>1$, denote the
set $\{\langle x_1, \ldots, x_m \rangle: x_i\in \mathbb{R}, \ i=1,\ldots,m \}$.
In the sequel, $\mathbb{R} $ is also written as $\mathbb{R}^1$.

Throughout this paper, we suppose that $X$ \emph{is a nonempty set and} $d$ \emph{is the metric on} $X$.
For simplicity,
we also use $X$
to denote the metric space $(X, d)$.

The metric $\overline{d}$ on $X \times [0,1]$ is defined
as follows: for $(x,\al), (y, \beta) \in X \times [0,1]$,
$$  \overline{d } ((x,\al), (y, \beta)) = d(x,y) + |\al-\beta| .$$

Throughout this paper, we suppose that \emph{the metric on} $X\times[0,1]$ \emph{is} $\overline{d}$.
For simplicity,
we also use $X \times [0,1]$
to denote the metric space $(X \times [0,1], \overline{d})$.

Let $m\in \mathbb{N}$.
For simplicity,
 $\mathbb{R}^m$ is also used to denote the
$m$-dimensional Euclidean space;
$d_m$ is used to denote the Euclidean metric on $\mathbb{R}^m$;
$\mathbb{R}^m \times [0,1]$ is also used to denote
the metric space $(\mathbb{R}^m \times [0,1], \overline{d_m})$.

A fuzzy set $u$ in $X$ can be seen as a function $u:X \to [0,1]$.
A
subset $S$ of $X$ can be seen as a fuzzy set in $X$. If there is no confusion,
 the fuzzy set corresponding to $S$ is often denoted by $\chi_{S}$; that is,
\[ \chi_{S} (x) = \left\{
                    \begin{array}{ll}
                      1, & x\in S, \\
                      0, & x\in X \setminus S.
                    \end{array}
                  \right.
\]
For simplicity,
for
$x\in X$, we will use $\widehat{x}$ to denote the fuzzy set  $\chi_{\{x\}}$ in $X$.
In this paper, if we want to emphasize a specific metric space $X$, we will write the fuzzy set corresponding to $S$ in $X$ as
$S_{F(X)}$, and the fuzzy set corresponding to $\{x\}$ in $X$ as $\widehat{x}_{F(X)}$.

The symbol $F(X)$ is used
to
denote the set of
all fuzzy sets in $X$.
For
$u\in F(X)$ and $\al\in [0,1]$, let $\{u>\al \} $ denote the set $\{x\in X: u(x)>\al \}$, and let $[u]_{\al}$ denote the \emph{$\al$-cut} of
$u$, i.e.
\[
[u]_{\al}=\begin{cases}
\{x\in X : u(x)\geq \al \}, & \ \al\in(0,1],
\\
{\rm supp}\, u=\overline{    \{ u > 0 \}    }, & \ \al=0,
\end{cases}
\]
where $\overline{S}$
denotes
the topological closure of $S$ in $(X,d)$.

The symbol $K(X)$ and
 $C(X)$ are used
to
 denote the set of all nonempty compact subsets of $X$ and the set of all nonempty closed subsets of $X$, respectively.

Let
$F_{USC}(X)$
denote
the set of all upper semi-continuous fuzzy sets $u:X \to [0,1]$,
i.e.,
$$F_{USC}(X) :=\{ u\in F(X) : [u]_\al \in  C(X) \cup \{\emptyset\}  \  \mbox{for all} \   \al \in [0,1]   \}.  $$

Define
\begin{gather*}
F_{USCB}(X):=\{ u\in  F_{USC}(X): [u]_0 \in K(X)\cup\{\emptyset\} \},
\\
F_{USCG}(X):=\{ u\in  F_{USC}(X): [u]_\al \in K(X)\cup\{\emptyset\} \ \mbox{for all} \   \al\in (0,1] \}.
 \end{gather*}
Clearly,
 $$F_{USCB}(X) \subseteq  F_{USCG}(X)   \subseteq  F_{USC}(X).$$

Define
\begin{gather*}
F_{CON}(X) := \{ u\in F(X): \mbox{for all } \al\in (0,1], \ [u]_\al \mbox{ is connected in } X   \},
\\
F_{USCCON} (X) :=  F_{USC}(X) \cap F_{CON}(X),\\
F_{USCGCON}(X):= F_{USCG}(X) \cap F_{CON}(X).
\end{gather*}

Let $u\in F_{CON}(X)$. Then $[u]_0 = \overline{\cup_{\al>0} [u]_\al}$ is connected in $X$.
The proof is as follows.

 If $u=\chi_{\emptyset}$, then $[u]_0 = \emptyset$ is connected in $X$.
 If $u \not= \chi_{\emptyset}$,
 then there is an $\al\in (0,1]$ such that $[u]_\al \not= \emptyset$.
Note that $[u]_\beta \supseteq [u]_\al$ when $\beta\in [0,\al]$.
Hence $\cup_{0<\beta<\al} [u]_\beta$ is connected, and thus
 $[u]_0 = \overline{\cup_{0<\beta<\al} [u]_\beta}$ is connected.

So
 $$
F_{CON}(X) = \{ u\in F(X): \mbox{for all } \al\in [0,1], \ [u]_\al \mbox{ is connected in } X   \}.
$$

Let
$F^1_{USC}(X)$
denote
the set of all normal and upper semi-continuous fuzzy sets $u:X \to [0,1]$,
i.e.,
$$F^1_{USC}(X) :=\{ u\in F(X) : [u]_\al \in  C(X)  \  \mbox{for all} \   \al \in [0,1]   \}.  $$

We introduce some subclasses of $F^1_{USC}(X)$, which will be discussed in this paper.
Define
\begin{gather*}
F^1_{USCB}(X):= F^1_{USC}(X) \cap F_{USCB}(X) ,
\\
F^1_{USCG}(X):= F^1_{USC}(X) \cap F_{USCG}(X),
\\
  F^1_{USCCON} (X) := F^1_{USC}(X) \cap F_{CON}(X), \\
  F^1_{USCGCON} (X) := F^1_{USCG} (X) \cap F_{CON} (X).
 \end{gather*}
Clearly,
\begin{gather*}
F^1_{USCB}(X) \subseteq  F^1_{USCG}(X)   \subseteq  F^1_{USC}(X),\\
 F^1_{USCGCON} (X) \subseteq F^1_{USCCON} (X).
\end{gather*}

Let
$(X,d)$ be a metric space.
 We
use $\bm{H}$ to denote the \emph{\textbf{Hausdorff distance}}
on
 $C(X)$ induced by $d$, i.e.,
\begin{equation} \label{hau}
\bm{H(U,V)}  =   \max\{H^{*}(U,V),\ H^{*}(V,U)\}
\end{equation}
for arbitrary $U,V\in C(X)$,
where
  $$
H^{*}(U,V)=\sup\limits_{u\in U}\,d\, (u,V) =\sup\limits_{u\in U}\inf\limits_{v\in
V}d\, (u,v).
$$

If
there is no confusion, we also use $H$ to denote the Hausdorff distance on $C(X\times [0,1])$ induced by $\overline{d}$.

\begin{re}
{\rm

$\rho$ is said to be a \emph{metric} on $Y$ if $\rho$ is a function from $Y\times Y$ into $\mathbb{R}$
satisfying
positivity, symmetry and triangle inequality. At this time, $(Y, \rho)$ is said to be a metric space.

  $\rho$ is said to be an \emph{extended metric} on $Y$ if $\rho$ is a function from $Y\times Y$ into $\mathbb{R} \cup \{+\infty\} $
satisfying
positivity, symmetry and triangle inequality. At this time, $(Y, \rho)$ is said to be an extended metric space.

We can see that for arbitrary metric space $(X,d)$, the Hausdorff distance $H$ on $K(X)$ induced by $d$ is a metric.
So
the Hausdorff distance $H$ on $K(X\times [0,1])$ induced by $\overline{d}$ on $X\times [0,1]$
is a metric. In these cases, we call the Hausdorff distance the Hausdorff metric.

The Hausdorff distance $H$ on $C(X)$ induced by $d$ on $X$
is an extended metric, but probably not a metric,
because
$H(A,B)$ could be equal to $+\infty$ for certain metric space $X$ and $A, B \in C(X)$.
Clearly, if $H$ on $C(X)$ induced by $d$
is not a metric, then $H$ on $C(X\times [0,1])$ induced by $\overline{d}$
is also not a metric.
So
the Hausdorff distance $H$ on $C(X\times [0,1])$ induced by $\overline{d}$ on $X\times [0,1]$
 is an extended metric but probably not a metric.
In the cases that the Hausdorff distance $H$ is an extended metric, we call the Hausdorff distance the Hausdorff extended metric.

We can see that $H$ on $C(\mathbb{R}^m)$ is an extended metric but not a metric,
and then the same is $H$ on $C(\mathbb{R}^m\times [0,1])$.

In this paper, for simplicity,
 we refer to both the Hausdorff extended metric and the Hausdorff metric as the Hausdorff metric.

}
\end{re}

For
$u\in F(X)$,
define
\begin{gather*}
{\rm end}\, u:= \{ (x, t)\in  X \times [0,1]: u(x) \geq t\},
\\
{\rm send}\, u:= \{ (x, t)\in  X \times [0,1]: u(x) \geq t\} \cap  ([u]_0 \times [0,1]).
\end{gather*}
 $
{\rm end}\, u$ and ${\rm send}\, u$
 are called the endograph and the sendograph of $u$, respectively.

Let $u \in  F(X)$. The following properties (\romannumeral1)-(\romannumeral3)
are
equivalent:
\\
(\romannumeral1) \ $u\in F_{USC} (X)$;
\\
(\romannumeral2) ${\rm end}\, u$ is closed in $(X\times [0,1],   \overline{d})$;
\\
(\romannumeral3) ${\rm send}\, u$ is closed in $(X\times [0,1],  \overline{d})$.

(\romannumeral1)$\Rightarrow$(\romannumeral2). Assume that (\romannumeral1) is true.
To show that (\romannumeral2) is true, let $\{(x_n, \al_n)\}$ be a sequence
in ${\rm end}\, u$ which converges to $(x,\al)$ in $X\times [0,1]$, we only need to show that $(x,\al) \in {\rm end}\, u$.
Since
 $u$ is upper semi-continuous, then
$u(x) \geq \limsup_{n\to\infty} u(x_n) \geq \lim_{n\to\infty} \al_n = \al$. Thus $(x,\al) \in {\rm end}\, u$.
So (\romannumeral2) is true.

(\romannumeral2)$\Rightarrow$(\romannumeral3).
Assume that (\romannumeral2) is true.
Note that $[u]_0 \times [0,1]$ is closed in $X\times [0,1]$,
then
${\rm send}\, u =  {\rm end}\, u  \cap ([u]_0 \times [0,1])$ is closed in $X\times [0,1]$.
So
(\romannumeral3) is true.

(\romannumeral3)$\Rightarrow$(\romannumeral1).
 Assume that (\romannumeral3) is true.
To show that (\romannumeral1) is true,
let $\al\in [0,1]$ and suppose that $\{x_n\}$ is
 a sequence in $[u]_\al$ which converges to $x$ in $X$,
 we only need to show that $x\in [u]_\al$.
Note that
$\{(x_n, \al)\}$ converges to $(x,\al)$ in $X\times[0,1]$,
and that
the sequence $\{(x_n,\al)\}$ is in ${\rm send}\, u$.
Hence from the closedness of ${\rm send}\, u$, it follows that $(x,\al) \in {\rm send}\, u$,
which means that $x\in [u]_\al$.
So
(\romannumeral1) is true.

Let $u\in F(X)$. Clearly $X\times\{0\} \subseteq {\rm end}\, u$. So
${\rm end}\, u \not=\emptyset$.
We can see that ${\rm send}\, u = \emptyset$ if and only if $u=\emptyset_{F(X)}$.

From the above discussions, we know that $u\in F_{USC}(X)$ if and only if
${\rm end}\, u\in C(X \times [0,1])$.

The endograph metric
$H_{\rm end}$
 on $F_{USC}(X)$ can be defined as usual.
For $u,v \in F_{USC}(X)$,
\begin{gather*}
\bm{  H_{\rm end}(u,v)    }: =  H({\rm end}\, u,  {\rm end}\, v ),
  \end{gather*}
where
 $H$
is
the Hausdorff
metric on $C(X \times [0,1])$ induced by $\overline{d}$ on $X \times [0,1]$.

Rojas-Medar and Rom\'{a}n-Flores \cite{rojas} introduced the Kuratowski convergence
of a sequence of sets in a metric space.

Let $(X,d)$ be a metric space.
Let $C$ be a set in $X$ and
$\{C_n\}$ a sequence of sets in $X$.
 $\{C_n\}$ is said to \emph{\textbf{Kuratowski converge}} to
$C$ according to $(X,d)$, if
$$
C
=
\liminf_{n\rightarrow \infty} C_{n}
=
\limsup_{n\rightarrow \infty} C_{n},
$$
where
\begin{gather*}
\liminf_{n\rightarrow \infty} C_{n}
 =
 \{x\in X: \  x=\lim\limits_{n\rightarrow \infty}x_{n},    x_{n}\in C_{n}\},
\\
\limsup_{n\rightarrow \infty} C_{n}
=
\{
 x\in X : \
 x=\lim\limits_{j\rightarrow \infty}x_{n_{j}},x_{n_{j}}\in C_{n_j}
\}
 =
 \bigcap\limits_{n=1}^{\infty}   \overline{   \bigcup\limits_{m\geq n}C_{m}    }.
\end{gather*}
In this case, we'll write
\bm{  $C=\lim^{(K)}_{n\to\infty}C_n    $ } according to $(X,d)$.
If
there is no confusion, we will not emphasize
the metric space
$(X,d)$ and write  $\{C_n\}$ \emph{\textbf{Kuratowski converges}} to
$C$ or \bm{  $C=\lim^{(K)}_{n\to\infty}C_n    $ } for simplicity.

\begin{re} \label{ksc}
{\rm
Definition 3.1.4 in \cite{kle} gives the definitions of
$\liminf C_{n}$, $\limsup C_{n}$
 and
$\lim C_{n}$
for
 a net of subsets $\{C_n, n\in D\}$ in a topological space.
When $\{C_n, n=1,2,\ldots\}$ is
 a sequence of subsets of a metric space,
$\liminf C_{n}$, $\limsup C_{n}$
 and
$\lim C_{n}$
according to Definition 3.1.4 in \cite{kle}
are
$\liminf_{n\rightarrow \infty} C_{n}$, $\limsup_{n\rightarrow \infty} C_{n}$
and $\lim^{(K)}_{n\to\infty}C_n  $
according to
 the above definitions, respectively.

  }
\end{re}

The $\Gamma$-convergence of a sequence
of fuzzy sets
on $F_{USC} (X)$ can be defined
 as usual.

Let $u$, $u_n$, $n=1,2,\ldots$, be fuzzy sets in $F_{USC} (X)$.
   $\{u_n\}$ is said to \bm{$\Gamma$}-\emph{\textbf{converge}}
  to
  $u$, denoted by \bm{$u = \lim_{n\to \infty}^{(\Gamma)}  u_n$},
  if
  ${\rm end}\, u= \lim_{n\to \infty}^{(K)}  {\rm end}\, u_n$
according to
$(X \times [0,1], \overline{d})$.

In this paper,
for a metric space $(Y,\rho)$ and a subset $S$ in $Y$,
we still use $\rho$ to denote the induced metric on $S$ by $\rho$.

\section{Representation theorems for various kinds of fuzzy sets}

In this section, we give representation theorems for various kinds of fuzzy sets.
These representation theorems
are useful in this paper.

The following representation theorem should be a known conclusion.
In this paper we assume that $\sup\emptyset = 0$.

\begin{tm} \label{rep}
Let $Y$ be a nonempty set.
If $u\in F(Y)$, then for all
$\al\in (0,1]$,
  $[u]_\al = \cap_{\beta<\al} [u]_\beta.$

Conversely,
suppose that
$\{v_\al: \al \in (0,1]\}$
is
a family of sets in $Y$ with $v_\al = \cap_{\beta<\al} v_\beta$ for all
$\al\in (0,1]$.
Define $u\in F(Y)$ by
$$u(x) := \sup\{  \al:  x\in v_\al    \} $$ for each $x\in Y$.
 Then $u$ is the unique fuzzy set in $Y$ satisfying that
 $[u]_\al = v_\al$ for all $\al\in (0,1]$.
\end{tm}

\begin{proof}
Let $u\in F(Y)$ and $\al\in (0,1]$. For each $x\in Y$,
$x\in [u]_\al \Leftrightarrow u(x) \geq \al \Leftrightarrow$ for each $\beta<\al$, $u(x) \geq \beta$
$\Leftrightarrow$ for each $\beta<\al$, $x\in [u]_\beta$.
So
 $[u]_\al = \cap_{\beta<\al} [u]_\beta.$

Conversely,
suppose that
$\{v_\al: \al \in (0,1]\}$
is
a family of sets in $Y$ with $v_\al = \cap_{\beta<\al} v_\beta$ for all
$\al\in (0,1]$.
Let $u\in F(Y)$ defined by
$$u(x) := \sup\{  \al:  x\in v_\al    \} $$ for each $x\in Y$.
Firstly, we show
that for each $\al\in (0,1]$,
$[u]_\al = v_\al$.
To do this, let $\al\in (0,1]$. We only need to verify that $[u]_\al \supseteq v_\al$ and $[u]_\al\subseteq v_\al$.

Let $x\in  v_\al $. Then clearly $u(x) \geq \al$, i.e.
$x\in [u]_\al$.
So $[u]_\al \supseteq v_\al$.

Let $x\in [u]_\al$. Then
$\sup\{ \beta:  x\in v_\beta    \} = u(x)\geq \al$.
Hence
there exists a sequence $\{ \beta_n, \ n=1,2,\ldots\}$ such that
$1\geq\beta_n \geq \al-1/n$ and $x\in v_{\beta_n}$.
Set $\gamma= \sup_{n=1}^{+\infty} \beta_n $. Then $1\geq \gamma\geq \alpha$
and thus $x\in \cap_{n=1}^{+\infty} v_{\beta_n} = v_\gamma \subseteq v_\al$.
So $[u]_\al\subseteq v_\al$.

Now we show the uniqueness of $u$. To do this, assume that
 $v$ is a fuzzy set in $Y$ satisfying that
 $[v]_\al = v_\al$ for all $\al\in (0,1]$. Then for each $x\in Y$,
$$
v(x) = \sup\{  \al:  x\in [v]_\al    \}  = \sup\{  \al:  x\in v_\al    \} = u(x).
$$
So $u=v$.

\end{proof}

\begin{re}
  {\rm

We can't find the original reference which gave Theorem \ref{rep}, so we give a proof here.
Theorem \ref{rep} and its proof are essentially the same as the Theorem 7.10 in P27 of chinaXiv:202110.00083v4
and its proof since
the uniqueness of $u$ is obvious.

}
\end{re}

From Theorem \ref{rep}, it follows immediately below representation theorems for
$F_{USC} (X)$, $F^1_{USC} (X)$, $F_{USCG} (X)$, $F_{CON} (X)$, $F_{USCB} (X)$,
and $F^1_{USCB} (X)$.

\begin{pp} \label{fus}\
  Let $(X,d)$ be a metric space.
If $u\in F_{USC} (X)$ (respectively, $u\in F^1_{USC} (X)$, $u\in F_{USCG} (X)$, $u\in F_{CON} (X)$), then
\\
(\romannumeral1) \ $[u]_\al\in C(X) \cup \{\emptyset\}$ (respectively, $[u]_\al\in C(X)$, $[u]_\al\in K(X) \cup \{\emptyset\}$, $[u]_\al$ is connected in $(X,d)$) for all $\al\in (0,1]$, and
\\
(\romannumeral2) \ $[u]_\al=\bigcap_{\beta<\al}[u]_\beta$ for all $\al\in (0,1]$.

Conversely, suppose that the family of sets $\{v_\al:\al\in (0,1]\}$ satisfies
conditions (\romannumeral1) and (\romannumeral2).
Define $u\in F(X)$ by
$u(x) := \sup\{  \al:  x\in v_\al    \} $ for each $x\in X$.
Then $u$ is the unique fuzzy set
in $X$
satisfying that $[u]_{\al}=v_\al$ for each $\al\in (0,1]$.
Moreover, $u\in F_{USC}(X)$ (respectively, $u\in F^1_{USC} (X)$, $u\in F_{USCG}(X)$, $u\in F_{CON}(X)$).

\end{pp}

\begin{proof}
  The proof is routine. We only show the case of $F_{USC} (X)$. The other cases can be verified similarly.

If $x\in F_{USC} (X)$, then clearly (\romannumeral1) is true.
From Theorem \ref{rep},
(\romannumeral2) is true.

Conversely, suppose that the family of sets $\{v_\al:\al\in (0,1]\}$ satisfies
conditions (\romannumeral1) and (\romannumeral2).
Define $u\in F(X)$ by
$u(x) := \sup\{  \al:  x\in v_\al    \} $ for each $x\in X$.
Then by Theorem \ref{rep}, $u$ is the unique fuzzy set
in $X$
satisfying that $[u]_{\al}=v_\al$ for each $\al\in (0,1]$.
Since $\{[u]_\al, \al\in (0,1] \}$ satisfies condition (\romannumeral1),
$u\in F_{USC}(X)$.

\end{proof}

\begin{pp} \label{fuscbchre}\
  Let $(X,d)$ be a metric space.
If $u \in F_{USCB} (X)$ (respectively, $u \in F^1_{USCB} (X)$), then
\\
(\romannumeral1) \ $[u]_\al\in K(X) \cup \{\emptyset\}$ (respectively, $[u]_\al\in K(X)$) for all $\al\in [0,1]$,
\\
(\romannumeral2) \ $[u]_\al=\bigcap_{\beta<\al}[u]_\beta$ for all $\al\in (0,1]$, and
\\
(\romannumeral3) \ $[u]_0=\overline{\bigcup_{\beta>0}[u]_\beta}$.

Conversely, suppose that the family of sets $\{v_\al:\al\in [0,1]\}$ satisfies
conditions (\romannumeral1) through (\romannumeral3).
Define $u\in F(X)$ by
$u(x) := \sup\{  \al:  x\in v_\al    \} $ for each $x\in X$.
Then $u$ is the unique fuzzy set
in $X$
satisfying that $[u]_{\al}=v_\al$ for each $\al\in [0,1]$.
Moreover, $u\in F_{USCB}(X)$
(respectively, $u \in F^1_{USCB} (X)$).

\end{pp}

\begin{proof}
    The proof is routine. We only show the case of $F_{USCB} (X)$. The case of $F^1_{USCB} (X)$ can be verified similarly.

If $x\in F_{USCB} (X)$, then clearly (\romannumeral1) is true.
By Theorem \ref{rep},
(\romannumeral2) is true.
From the definition of $[u]_0$, (\romannumeral3) is true.

Conversely, suppose that the family of sets $\{v_\al:\al\in [0,1]\}$ satisfies
conditions (\romannumeral1) through (\romannumeral3).
Define $u\in F(X)$ by
$u(x) := \sup\{  \al:  x\in v_\al    \} $ for each $x\in X$.
Then by Theorem \ref{rep}, $u$ is the unique fuzzy set
in $X$
satisfying that $[u]_{\al}=v_\al$ for each $\al\in (0,1]$.
Clearly $[u]_0 = \overline{\bigcup_{\beta>0}[u]_\beta} = \overline{\bigcup_{\beta>0}v_\beta} = v_0$.
Since $\{[u]_\al, \al\in [0,1] \}$ satisfies condition (\romannumeral1), $u\in F_{USCB}(X)$.

\end{proof}

Similarly,
we can obtain the
representation theorems for $F_{USCCON}(X)$, $F_{USCGCON}(X)$, $F^1_{USCCON}(X)$, etc.

Based on these representation theorems, we can define a fuzzy set or a certain type fuzzy set
by giving the family of its $\al$-cuts.
For instance, in Example \ref{ucn}, we define a fuzzy set $u$ by giving the family of its $\al$-cuts.
 By Proposition \ref{fus}, we can obtain that this $u$
belongs to
 $F^1_{USC}(\mathbb{R})$.
In the sequel, as done in Example \ref{ucn},
we will directly point out that what we defined is a fuzzy set or a certain type fuzzy set without saying which  representation theorem is used since it is easy to see.

\section{Compatibility of
 endograph metric $H_{\rm end}$ and $\Gamma$-convergence in a special case} \label{nce}

In this section, using the level characterizations of the endograph metric $H_{\rm end}$ and
the $\Gamma$-convergence given in \cite{huang, huang719},
 we give an improvement of the results in \cite{huang} on compatibility of
the endograph metric $H_{\rm end}$ and the $\Gamma$-convergence in a special case.
It is shown that the prerequisite can be relaxed
by deleting the condition of the compactness of the positive $\al$-cuts
of the convergent sequence.

We suspect that
the following Theorem \ref{hkg}
is an already known conclusion, however we can't find this conclusion in the references that we can obtain.
  It can be proved in a similar fashion to Theorem 4.1 in \cite{huang}.
In this paper, we exclude the case that $C=\emptyset$.

\begin{tm} \cite{huang} \label{hkg}
Suppose that $C$, $C_n$ are sets in $C(X)$, $n=1,2,\ldots$. Then $H(C_n, C) \to 0$ implies that $\lim_{n\to \infty}^{(K)} C_n \, =C$.
\end{tm}

\begin{re} \label{hmr}
  {\rm

Theorem \ref{hkg} implies that
for a sequence
$\{u_n\}$ in $F_{USC}(X)$ and an element $u$ in $F_{USC}(X)$,
if
$H_{\rm end} (u_n, u) \to 0$ as $n \to \infty$, then $\lim_{n\to\infty}^{(\Gamma)} u_n = u$.
However,
the converse is false. See Example 4.1 in \cite{huang}.
}
\end{re}

We \cite{huang} found
that
the endograph metric $H_{\rm end}$ and the $\Gamma$-convergence
 is compatible
under certain prerequisite.
The conclusion is listed in the following.

\begin{tm} (Theorem 9.2 in \cite{huang}) \label{com} \
Let $u$ be a fuzzy set in $F_{USCG} (\mathbb{R}^m) \setminus \{\emptyset_{F(\mathbb{R}^m)}\}$
and for $n=1,2,\ldots$,
let $u_n$ be a fuzzy set in $F_{USCGCON} (\mathbb{R}^m)$.
Then
$H_{\rm end} (u_n, u) \to 0$ as $n \to \infty$ if and only if $\lim_{n\to\infty}^{(\Gamma)} u_n = u$.
\end{tm}

The following Corollary \ref{fer} is an immediate corollary of Theorem 9.2 in \cite{huang}.

\begin{tl} \label{fer}
  Let $u$ be a fuzzy set in $F^1_{USCG} (\mathbb{R}^m)$
and for $n=1,2,\ldots$,
let $u_n$ be a fuzzy set in $F^1_{USCGCON} (\mathbb{R}^m)$.
Then
$H_{\rm end} (u_n, u) \to 0$ as $n \to \infty$ if and only if $\lim_{n\to\infty}^{(\Gamma)} u_n = u$.
\end{tl}
Corollary \ref{fer} is
the normal fuzzy set case of Theorem 9.2 in \cite{huang} (Theorem \ref{com} in this paper),
which is an important special case of
 Theorem 9.2 in \cite{huang}.

In this section, we give an improvement of Corollary \ref{fer}.
 The following conclusions are useful.

\begin{lm} (Lemma 2.1 in \cite{huang}) \label{infe}
Let $(X,d)$ be a metric space,
and
$C_{n}$, $n=1,2,\ldots$, be a sequence of sets in $X$.
Suppose that $x\in X$.
Then
\\
(\romannumeral1) \
  $x \in \liminf_{n\rightarrow \infty} C_{n}$
  if and only if
$\lim_{n\to \infty} d(x, C_n) = 0$,
\\
(\romannumeral2) \
  $x \in \limsup_{n\rightarrow \infty} C_{n}$
  if and only if there is a subsequence $\{C_{n_k}\}$ of $\{C_n\}$
such
that
$\lim_{k\to \infty} d(x, C_{n_k}) = 0$.
\end{lm}

\begin{proof} Here we give a detailed proof. Readers who think that this conclusion is
obvious can skip this proof.

(\romannumeral1) \ Assume that $x\in \liminf_{n\to \infty} C_n$. Then there is a sequence $\{x_n, n=1,2,\ldots\}$ in $X$
 such that $x_n\in C_n$ for $n=1,2,\ldots$ and $\lim_{n\to \infty} d(x, x_n) = 0$.
Since $d(x, C_n) \leq d(x, x_n)$, thus $\lim_{n\to \infty}  d(x, C_n)  = 0$.

Conversely, assume that $\lim_{n\to \infty} d(x, C_n) = 0$. For each $n=1,2,\ldots$, we can choose an $x_n$ in $C_n$
such that $d(x, x_n) \leq d(x, C_n) + 1/n$.
Hence
$\lim_{n\to\infty} d(x, x_n) = 0$. So $x\in \liminf_{n\to \infty} C_n$.

(\romannumeral2) \ Assume that $x\in \limsup_{n\to \infty} C_n$.
Then
there is a subsequence $\{C_{n_k}\}$ of $\{C_n\}$
and $x_{n_k}\in C_{n_k}$ for $k=1,2,\ldots$ such that $\lim_{k\to \infty} d(x, x_{n_k}) = 0$.
Since
$d(x, C_{n_k}) \leq d(x, x_{n_k})$, thus $\lim_{k\to \infty} d(x, C_{n_k})  = 0$.

Conversely, assume that
 there is a subsequence $\{C_{n_k}\}$ of $\{C_n\}$ such that
$\lim_{k\to \infty} d(x, C_{n_k})  = 0$.
For each $k=1,2,\ldots$, we can choose an $x_{n_k}$ in $C_{n_k}$
such that $d(x, x_{n_k}) \leq d(x, C_{n_k}) + 1/k$.
Hence
$\lim_{k\to\infty} d(x, x_{n_k}) = 0$.
So $x\in \limsup_{n\to \infty} C_n$.

\end{proof}

\begin{tm} (Theorem 5.19 in \cite{huang719}) \label{aec}
Let $u$ be a fuzzy set in $F^1_{USCG} (X)$ and let $u_n$, $n=1,2,\ldots$, be fuzzy sets in $F^1_{USC} (X)$.
Then
the following are equivalent:
\begin{enumerate}
\renewcommand{\labelenumi}{(\roman{enumi})}

\item
  $H_{\rm end}(u_n, u) \to 0$;

  \item
 $H([u_n]_\al, [u]_\al) \rightarrow 0$ holds a.e. on $\al\in (0,1)$;

\item     $H ([u_n]_\al, [u]_\al) \to 0 $ for all $\al\in (0,1) \setminus P_0(u)$;

    \item   There is a dense subset $P$ of $(0,1) \backslash P_0(u)$ such that $H ([u_n]_\al, [u]_\al) \to 0 $ for
$\al\in P$;

\item     There is a countable dense subset $P$ of $(0,1) \backslash P_0(u)$
such that
 $H ([u_n]_\al, [u]_\al) \to 0 $ for
$\al\in P$.

  \end{enumerate}

\end{tm}

\begin{tm} (Theorem 6.2 in \cite{huang})
 \label{gdmc}
Let $u$, $u_n$, $n=1,2,\ldots$ be fuzzy sets
  in $F_{USC} ( \mathbb{R}^m )$. Then the following are equivalent:
\begin{enumerate}
  \renewcommand{\labelenumi}{(\roman{enumi})}

    \item $\lim_{n\to \infty}^{(\Gamma)}  u_n = u $;

      \item $\lim^{(K)}_{n\to \infty} [u_n]_\al = [u]_\al$ holds a.e. on $\al \in (0,1)$;

    \item    $\lim^{(K)}_{n\to \infty} [u_n]_\al = [u]_\al$ holds for all $\al\in (0,1) \setminus P(u)$;

    \item     There is a dense subset $P$ of $(0,1) \backslash P(u)$ such that $\lim^{(K)}_{n\to \infty} [u_n]_\alpha = [u]_\al$ holds for
$\al\in P$;

 \item There is a countable dense subset $P$ of $(0,1) \backslash P(u)$ such that $\lim^{(K)}_{n\to \infty} [u_n]_\alpha = [u]_\al$ holds for $\al\in P$.

\end{enumerate}

\end{tm}

\begin{pp} \label{cne}
Let $C$ be a nonempty compact set in $\mathbb{R}^m$ and for $n=1,2,\ldots$ let $C_n$
be a nonempty connected and closed set in $\mathbb{R}^m$.
Then
$H(C_n, C) \to 0$ if and only if $\lim_{n\to\infty}^{(K)} C_n =C$.
 \end{pp}

\begin{proof}
  From Theorem \ref{hkg}, we have that $H(C_n, C) \to 0  \Rightarrow \lim_{n\to\infty}^{(K)} C_n =C$.

Now we show that $\lim_{n\to\infty}^{(K)} C_n =C \Rightarrow H(C_n, C) \to 0$.
We prove by contradiction.
Assume that $\lim_{n\to\infty}^{(K)} C_n =C$ but $H(C_n, C) \not\to 0$.
Then
$H^*(C, C_n) \not\to 0$ or $H^*( C_n, C) \not\to 0$.
We split the proof into two cases.

Case (\romannumeral1) $H^*(C, C_n) \not\to 0$.

In this case, there is an $\varepsilon>0$ and a subsequence $\{C_{n_k}\}$ of $\{C_n\}$
such that
$H^*(C, C_{n_k}) > \varepsilon$.
Thus for each $k=1,2,\ldots$ there exists
$x_k \in C$ such that $d(x_k, C_{n_k}) > \varepsilon$.
Since $C$ is compact,
there is a subsequence $\{x_{k_l}\}$ of $\{x_k\}$ which converges to $x\in C$.
Then
there is a $L(\varepsilon)$ such that
$d(x, x_{k_l}) < \varepsilon/2$ for all $l\geq L$. Hence
$d(x, C_{n_{k_l}}) \geq d(x_{k_l}, C_{n_{k_l}}) - d(x, x_{k_l}) > \varepsilon/2$. By Lemma \ref{infe} (\romannumeral1), this contradicts $x\in C = \liminf_{n\to\infty} C_n$.

Case (\romannumeral2) $H^*( C_n, C) \not\to 0$.

In this case,
 there is an $\varepsilon>0$ and a subsequence $\{C_{n_k}\}$ of $\{C_n\}$
such that
$H^*(C_{n_k}, C) > \varepsilon$.
Thus we have the following:
\\
(a) for each $k=1,2,\ldots$ there exists
$x_{n_k} \in C_{n_k}$ such that $d(x_{n_k}, C) > \varepsilon$.

Pick $y\in C$, since $C = \liminf_{n\to\infty} C_n$, we can find a sequence $\{y_n\}$
satisfying that
$y_n \in C_n$ for
$n=1,2,\ldots$
and
$\{y_n\}$ converges to $y$.
Hence
 there is a $N (\varepsilon)$ such that for all $n\geq N$,
$d(y_n, y) < \varepsilon$.
Since $d(y_n, C) \leq d(y_n, y) $,
we have the following:
\\
 (b) for all $n\geq N$,
$d(y_n, C) < \varepsilon$.

Let $k\in \mathbb{N}$ with $n_k \geq N$.
Define a function $f_k$ from $C_{n_k}$ to $\mathbb{R}$ given by $f_k (z)=d(z, C)$ for each $z \in C_{n_k}$.
 Then $f_k$  is a continuous function on $C_{n_k}$. Since
$C_{n_k}$ is a connected set in $\mathbb{R}^m$,
$f_k (C_{n_k})$
 is a connected set in $\mathbb{R}$.
Combined this fact with the above
clauses (a) and (b), we obtain
that there exists $z_{n_k} \in C_{n_k}$ such that
\begin{equation}\label{eun}
d(z_{n_k}, C) = \varepsilon.
\end{equation}
From \eqref{eun} and the compactness of
$C$, the set
 $\{   z_{n_k}, n_k \geq N   \}$ is bounded in $\mathbb{R}^m$, and thus $\{ z_{n_k}, n_k \geq N   \}$ has a cluster point $z$.
By \eqref{eun}, $d(z, C) = \varepsilon$, which contradicts
$z\in \limsup_{n\to\infty} C_n = C$.

\end{proof}

\begin{re}
{\rm

Proposition \ref{cne}
may be known, however we can't find this conclusion in the references that we can
obtain. So we give a proof here.

Let $A$ be a nonempty compact set in $\mathbb{R}^m$ and $B$ a nonempty closed set in $\mathbb{R}^m$.
If $H(A,B) < +\infty$, then $B$ is bounded and hence a compact set in $\mathbb{R}^m$.
Clearly, if $B$ is compact in $\mathbb{R}^m$, then $H(A,B) < +\infty$.
So
 $H(A,B) < +\infty$ if and only if $B$ is a compact set in $\mathbb{R}^m$.

From the above fact we know that
 for $C$ and $C_n$, $n=1,2,\ldots$, satisfying the
assumptions
of Proposition \ref{cne},
if $H(C_n, C) \to 0$ (by Proposition \ref{cne} $H(C_n, C) \to 0$
if and only if $\lim_{n\to\infty}^{(K)} C_n =C$),
  then clearly there is an $N \in \mathbb{N}$ such that for all $n\geq N$,
$H(C_n, C) < +\infty$ and thus for all $n\geq N$, $C_n$ is compact.
}
\end{re}

The following Theorem \ref{fce} is the main result of this section,
which
is an immediate consequence of
Proposition \ref{cne}, and the level characterizations of the endograph metric $H_{\rm end}$ and
the $\Gamma$-convergence (Theorems \ref{aec} and \ref{gdmc}).

 Theorem \ref{fce}
 shows that
the condition of the compactness of the positive $\al$-cuts
of the sequence
$\{u_n\}$ in Corollary \ref{fer} can be deleted.
So clearly
Corollary \ref{fer}
is a
corollary of Theorem \ref{fce}.

\begin{tm} \label{fce}
Let $u$ be a fuzzy set in $F^1_{USCG} (\mathbb{R}^m)$
and for $n=1,2,\ldots$,
let $u_n$ be a fuzzy set in $F^1_{USCCON} (\mathbb{R}^m)$.
Then
$H_{\rm end} (u_n, u) \to 0$ as $n\to\infty$ if and only if $\lim_{n\to\infty}^{(\Gamma)} u_n = u$.
\end{tm}

\begin{proof}

The proof is routine.

By let $X= \mathbb{R}^m$ in Theorem \ref{aec} we have the following:
\\
(\romannumeral1) $H_{\rm end} (u_n, u) \to 0$ if and only if
  $H ([u_n]_\al, [u]_\al) \to 0 $ holds a.e. on $\al\in (0,1)$.

By Theorem \ref{gdmc}, we have
the following:
\\
 (\romannumeral2)
$\lim_{n\to \infty}^{(\Gamma)}  u_n = u $
if and only if
  $[u]_\al=\lim^{(K)}_{n\to \infty} [u_n]_\al $ holds a.e. on $\al\in (0,1)$.

Since
for each $\al\in (0,1]$ and $n\in \mathbb{N}$, $[u]_\al \in K(\mathbb{R}^m)$,  $[u_n]_\al \in C(\mathbb{R}^m)$ and $[u_n]_\al$ is connected in $\mathbb{R}^m$.
Thus by Proposition \ref{cne}, for each
 $\alpha \in (0,1]$,
$H ([u_n]_\al, [u]_\al) \to 0$
if and only if
$[u]_\al=\lim^{(K)}_{n\to \infty} [u_n]_\al $.
Combined this fact with the above clauses (\romannumeral1) and (\romannumeral2),
we have
that
$H_{\rm end} (u_n, u) \to 0$ if and only if $\lim_{n\to \infty}^{(\Gamma)}  u_n = u$.

\end{proof}

\begin{re}
{\rm
Theorem \ref{fce} would be false if $\mathbb{R}^m$
were replaced by a general metric space $X$.

}
\end{re}

Here we mention that
for $u \in F^1_{USCG}(\mathbb{R})$ and a sequence $\{u_n\}$
in $F^1_{USCCON}(\mathbb{R})$,
$H_{\rm end} (u_n, u) \to 0$ does not imply that
there is an $N$
satisfying that for all $n\geq N$,
$u_n \in F^1_{USCGCON}(\mathbb{R})$.

The following Example \ref{mke} is a such example, which shows
that
there exists a
$u \in F^1_{USCG}(\mathbb{R})$ and a sequence $\{u_n\}$
in $F^1_{USCCON}(\mathbb{R})$
such that
\\
(\romannumeral1)
$H_{\rm end} (u_n, u) \to 0$,
\\
(\romannumeral2) for each $n=1,2,\ldots$, $u_n\notin F^1_{USCGCON}(\mathbb{R})$.

\begin{eap} \label{mke}
  {\rm
Let $u= \widehat{1}_{F(\mathbb{R})} \in F^1_{UCCG}(\mathbb{R})$.
For $n=1,2,\ldots$,
define $u_n \in F^1_{USC}(\mathbb{R})$ as follows:
\[u_n (t) = \left\{
              \begin{array}{ll}
                1, & t=1,\\
              1/n, & t\not=1.
              \end{array}
            \right.
\]
Then for $n=1,2,\ldots$,
\[ [u_n]_\al = \left\{
              \begin{array}{ll}
                \{1\}, & \al \in (1/n,1],\\
              \mathbb{R}, & \al\in [0,1/n].
              \end{array}
            \right.
\]
So for each $n=1,2,\ldots$, $u_n \in  F^1_{USCCON}(\mathbb{R})$ but $u_n \notin  F^1_{USCGCON}(\mathbb{R})$.
It can be seen that $H_{\rm end}(u, u_n) = 1/n \to 0$.

}
\end{eap}

We can see that Theorem \ref{fce} is an improvement of Corollary \ref{fer}.

\section{Compatibility of
 endograph metric $H_{\rm end}$ and $\Gamma$-convergence on general fuzzy sets}
\label{nceu}

In this section, we
give a general conclusion on compatibility of the endograph metric $H_{\rm end}$ and the $\Gamma$-convergence.
This conclusion applies to the general fuzzy sets,
which are allowed to not satisfy the normality.
 This conclusion
greatly improves both the corresponding result in \cite{huang} (Theorem \ref{com} in this paper)
and the result in Section \ref{nce}.
This conclusion
shows that
not only
the condition of the compactness of the positive $\al$-cuts
of the convergent sequence
can be deleted,
but also
the condition of the connectedness
of
  all the positive $\al$-cuts
of the convergent sequence
 can be replaced by a much weaker condition introduced in this section, which is called the connectedness condition.
Compared with the corresponding result in \cite{huang},
we adopt a new proof idea
 to show the conclusion in this section.

Below well-known conclusions are useful in this paper.
\begin{itemize}
  \item[]
  Let $(Z,\tau_1)$ and $(Y, \tau_2)$ be topological space, and let
$f$ be a continuous function from $(Z,\tau_1)$ to $(Y, \tau_2)$.

\item[] (\romannumeral1) \ If $A$ is compact in $(Z,\tau_1)$, then $f(A)$ is compact in $(Y, \tau_2)$.

\item[] (\romannumeral2) \ If $A$ is connected in $(Z,\tau_1)$, then $f(A)$ is connected in $(Y, \tau_2)$.
\end{itemize}

Let $p_X$ be the projection mapping from $(X\times [0,1], \overline{d})$ to $(X, d)$;
that is,
$p_X (x,\al) =x$ for each $(x,\al) \in  X\times [0,1]$.
Clearly for $(x,\al), (y,\beta)\in  X\times [0,1]$,
$d(p_X(x,\al), p_X(y,\beta)) = d(x,y) \leq \overline{d} ( (x,\al), (y,\beta)) $.
So $p_X$ is continuous.

For a set $S$ in $ X$, we use $\overline{S}$
to denote the topological closure of
$S$ in $(X,d)$;
for a set $S$ in $ X\times [0,1]$, we use $\overline{S}$
to denote the topological closure of
$S$ in $(X\times [0,1], \overline{d})$.
The readers can judge
the meaning of
$\overline{S}$
according to the contexts.

First, we give some conclusions on connectedness of a set in $X \times [0,1] $.

For $D \subseteq X \times [0,1] $ and $\al\in [0,1]$,
define
$\bm{D_\al} := \{x\in X:  (x,\al) \in D\}$.
For a nonempty set $D$ in $X \times [0,1] $,
define
$\bm{f_D} := \inf\{ \al:  (x,\al) \in D    \}$
and
$\bm{S_D} := \sup\{ \al:  (x,\al) \in D    \}$.

\begin{pp} \label{cptm}
  Let
 $D$ be a nonempty set in $X\times [0,1]$ with
$D_r \subseteq D_t$ for $f_D \leq t \leq r \leq 1$.
Then
 $D_{f_D}$ is connected in $(X,d)$
if and only if
 $D$ is connected in $(X \times [0,1], \overline{d})$.
\end{pp}

\begin{proof}
\textbf{Sufficiency}.
    We proceed by contradiction. Assume that $D_{f_D}$ is connected in $X$.
If
$D$ is not connected in $X \times [0,1]$,
then there exists two nonempty sets
$A$ and $B$ in $X \times [0,1]$ such that
$A \cup B = D$,
$A \cap \overline{B} = \emptyset$
and
 $B \cap \overline{A} = \emptyset$.

Note that $D_{f_D} \times \{f_D\}   \subseteq D$ and $D_{f_D} \times \{f_D\}$ is connected.
Hence
$D_{f_D} \times \{f_D\} \subseteq A$ or $D_{f_D} \times \{f_D\} \subseteq B$.
Without loss
of
generality, we suppose that $D_{f_D} \times \{f_D\} \subseteq A$.

Pick $(x,\al) \in B$. Set $\gamma = \inf\{\beta:   (x,\beta) \in B   \}$.

If
$(x, \gamma) \in B$, we claim that $\gamma > f_D$.
Otherwise $\gamma = f_D$
and
$(x, f_D) \in B$.
Note
that $(x, f_D) \in A$.
Thus  $A\cap B \not= \emptyset$, which is a contradiction.
Hence
$(x, \xi) \in A$ for $\xi\in [f_D, \gamma)$,
and therefore
$(x, \gamma) = \lim_{\xi\to \gamma-} (x,\xi)\in \overline{A}$.
So
$\overline{A}\cap B \not= \emptyset$. This is a contradiction.

If
$(x, \gamma) \in A$, then there is a sequence $\{(x, \gamma_n)\}$ in $B$
such that
$\lim_{n\to\infty} \gamma_n = \gamma$.
Hence $(x, \gamma) = \lim_{n\to \infty} (x, \gamma_n) \in \overline{B}$.
Thus
$A \cap \overline{B} \not= \emptyset$. This is a contradiction.

\vspace{3mm}
\textbf{Necessity}.
Assume that
$D$ is connected in $X \times [0,1]$.
Since $p_X$ is continuous, hence
$D_{f_D} = p_X(D)$  is connected in $X$.

\end{proof}

\begin{tl}
  \label{cptmeg}
  Let
 $E$ be a nonempty set in $X\times [0,1]$ with
$E_r \supseteq E_t$ for $0 \leq t \leq r \leq S_E$.
Then
 $E_{S_E}$ is connected in $(X,d)$
if and only if
 $E$ is connected in $(X \times [0,1], \overline{d})$
\end{tl}

\begin{proof}
  Let $D=\{(x,1-\al): (x,\al) \in E\}$. Then $D$ is a nonempty set in $X\times [0,1]$ with
$D_r \subseteq D_t$ for $f_D \leq t \leq r \leq 1$. Hence by Proposition \ref{cptm},
 $D_{f_D}$ is connected in $(X,d)$
if and only if
 $D$ is connected in $(X \times [0,1], \overline{d})$.

Define $f: (E, \overline{d}) \to (D, \overline{d})$ as follows: $f(x,\al) = (x, 1-\al)$ for $(x,\al) \in E$.

Observe that $E_{S_E} = D_{f_D}$,
so to verify the desired result it suffices to show that
$D$ is connected in $(X \times [0,1], \overline{d})$ if and only if $E$ is connected
in $(X \times [0,1], \overline{d})$, which
follows
from that $f$ is an isometry and $f(E) =D$.

The desired conclusion can also be proved in a similar manner as that of Proposition \ref{cptm}.

\end{proof}

\begin{re}
  {\rm
Here we mention that $D_{S_D} = \emptyset$ is possible when $D$ is connected in $X\times [0,1]$ and satisfies the assumption in Proposition \ref{cptm}.
Let $X=\mathbb{R}$ and
define $D \subset \mathbb{R} \times [0,1]$ by putting
\[D_\al =\left\{
    \begin{array}{ll}
  \emptyset, & \al=1,\\
(0, 1-\alpha], & \al\in [0,1).
      \end{array}
  \right.
\]
Then $D$ is a such example.

Similarly, $E_{f_E} = \emptyset$ is possible when $E$ is connected in $X\times [0,1]$ and satisfies the assumption in Corollary \ref{cptmeg}.
}
\end{re}

Let $u\in F(X)$ and $0\leq r \leq t\leq 1$.
We use the symbol $\bm{{\rm end}_r^t\, u }$ to denote
the subset of ${\rm end}\, u $ given by
$$
{\rm end}_r^t\, u  :=    {\rm end}\, u \cap  ([u]_r \times [r, t] ).$$
For simplicity, we write
${\rm end}_r^1\, u  $ as ${\rm end}_r\, u  $.
We can see that ${\rm end}_0\, u  = {\rm send}\, u$.

\begin{tl} \label{cdpu}
    Let $u\in  F(X)$.
\item (\romannumeral1) For $r,t$ with $0 \leq r\leq t \leq 1$,
${\rm end}_{r}^t\, u$ is connected in $X \times [0,1]$ if and only if $[u]_r$ is connected in $X$.

\item (\romannumeral2)
$X$ is connected
if and only if
${\rm end}\, u$ is connected in $X \times [0,1]$.
\end{tl}

\begin{proof}

If ${\rm end}_{r}^t\, u \not= \emptyset$, then, by Proposition \ref{cptm}, the conclusion in (\romannumeral1) is true.
Thus it suffices to consider
the case when
 ${\rm end}_{r}^t\, u = \emptyset$.
In this case,
 $[u]_r = \emptyset$, and
so clearly the conclusion in
(\romannumeral1) is true.

Note that $f_{{\rm end}\, u} = 0$ and $({\rm end}\, u)_0 = X$. So
(\romannumeral2) follows immediately
from Proposition \ref{cptm}.

\end{proof}

The following Examples \ref{dcn} and \ref{ucn} give some connected sets in $\mathbb{R} \times [0,1]$.
 Proposition \ref{cptm}, Corollary \ref{cptmeg} and Corollary \ref{cdpu} are used
to show the connectedness of these sets.

\begin{eap} \label{dcn}
  {\rm
Let $D \subset \mathbb{R} \times [0,1]$ be defined by putting
\[
D_\al=
\left\{
  \begin{array}{ll}
   [1-\alpha^2, 1] \cup [3, 4-\al^2], & \al \in (0.5, 1],\\
  \mbox{} [0,4], &  \al\in [0,0.5].
  \end{array}
\right.
\]
We can see that
$D= A \cup B \cup C$, where
\begin{gather*}
  A = [0,4] \times [0,0.5],\\
B=  \bigcup_{\al\in [0.5, 1]}  [1-\alpha^2, 1] \times \{\al\},\\
C= \bigcup_{\al\in [0.5, 1]}  [3, 4-\al^2] \times \{\al\}.
\end{gather*}
Clearly $A$ is connected in $\mathbb{R}\times [0,1]$.
By Corollary \ref{cptmeg}, $B$ is connected in $\mathbb{R}\times [0,1]$.
By Proposition \ref{cptm}, $C$ is connected in $\mathbb{R}\times [0,1]$.
$A \cap B =  [1-0.5^2, 1] \times \{0.5\}\not=\emptyset$.
$A \cap C =  [3, 4-0.5^2] \times \{0.5\}\not=\emptyset$.
Thus
$D$ is connected in  $\mathbb{R}\times [0,1]$.

Here we mention that there is no $u\in F(\mathbb{R})$ satisfying $D = {\rm send}\, u$
because
$D_1 \nsubseteqq D_{0.9}$.

}
\end{eap}

\begin{eap}\label{ucn}
  {\rm
Let $u\in F^1_{USC}(\mathbb{R})$ be defined by putting:
\[
[u]_\al=
\left\{
  \begin{array}{ll}
   [0,1] \cup [3,4], & \al \in (0.6, 1],\\
  \mbox{} [0,4], &  \al\in [0,0.6].
  \end{array}
\right.
\]

We can see that
 $[0,1] \cup [3,4]$ is not connected,
[0,4] and $\mathbb{R}$ are connected.
So
 $u\in F^1_{USC}(\mathbb{R})  \setminus   F_{USCCON}(\mathbb{R})  $.
By Corollary \ref{cdpu},
${\rm end}\, u$ is connected in $\mathbb{R} \times [0,1]$;
${\rm end}_{r}\, u$ is connected in $\mathbb{R} \times [0,1]$ if and only if $r\in  [0,0.6]$.

}
\end{eap}

\begin{pp}\label{cmpugn}
Let $D$ be a subset of $X\times [0,1]$.
\\
(\romannumeral1) \ If
$D$ is compact in $X\times [0,1]$, then $p_X(D)$ is compact in $X$.
\\
(\romannumeral2) \ If $D$ is closed in $X\times [0,1]$ and $p_X(D)$ is compact in $X$,
then $D$ is compact in $X\times [0,1]$.
\\
(\romannumeral3) \ Let $D$ be closed in $X\times [0,1]$. Then
$D$ is compact in $X\times [0,1]$ if and only if $p_X(D)$ is compact in $X$.
\end{pp}

\begin{proof}
  (\romannumeral1) holds obviously.

To show (\romannumeral2), let $D$ be a closed set in $X\times [0,1]$ with
$p_X(D)$ being compact in $X$.
If $D=\emptyset$, then clearly $D$ is compact in $X\times [0,1]$.
If $D\not=\emptyset$,
then
$p_X(D) \times [f_D, S_D]$ is compact in $X\times [0,1]$.
Since $D$ is closed in $X\times [0,1]$ and $D\subseteq p_X(D) \times [f_D, S_D]$,
it follows that
$D$ is compact in $X\times [0,1]$. So (\romannumeral2) is true.

  (\romannumeral3) is an immediate consequence of (\romannumeral1) and (\romannumeral2).

\end{proof}

\begin{tl}\label{cmpu}
Let $u\in F_{USC}(X)$ and $0 \leq r\leq t \leq 1$. Then ${\rm end}_r^t\, u$ is compact in $X\times [0,1]$
if and only if
$[u]_r$ is compact in $X$.
\end{tl}

\begin{proof}
  We can see that ${\rm end}_r^t\, u = {\rm end}\, u \cap ([u]_r \times [r,t])$ is closed in $X\times [0,1]$,
and
$p_X({\rm end}_r^t\, u) = [u]_r$.
So the desired result follows immediately from Proposition \ref{cmpugn} (\romannumeral3).

\end{proof}

By Corollary \ref{cmpu},
for $u\in F_{USCG}(X)$ and $r \in (0, 1]$,
${\rm end}_r\, u$ is a compact set in $X \times [0,1]$.

We will use the following conclusion.

Let $(Y,\rho)$ be a metric space, $x,y\in Y$ and $W\subseteq Y$.
Then
\begin{align} \label{uen}
    \rho(x,W) &  =\inf_{z\in W}   \rho(x,z)  \nonumber \\
&  \leq \inf_{z\in W}  \{\rho(x,y) + \rho(y,z) \} \nonumber
\\
&=\rho(x,y) + \rho(y,W).
        \end{align}

We say that a sequence $\{u_n, n=1,2,\ldots\}$ in $F_{USC} (\mathbb{R}^m) $ satisfies \textbf{\emph{connectedness condition}}
if for each $\varepsilon>0$, there is a $\delta \in (0,\varepsilon]$ and $N(\varepsilon) \in \mathbb{N}$
such that for all $n\geq N$,
$ {\rm end}_{\delta}\, u_n$ is connected in $\mathbb{R}^m \times [0,1]$.

\begin{re}
  {\rm

From clause (\romannumeral1) of Corollary \ref{cdpu},
it follows that
   for $u\in  F(X)$ and
 $r\in [0, 1]$,
${\rm end}_{r}\, u$ is connected in $X \times [0,1]$ if and only if $[u]_r$ is connected in $X$.
So
for a sequence $\{u_n: n\in \mathbb{N}\}$ in $F_{USC}(\mathbb{R}^m)$,
the condition that $\{u_n\}$ is in $F_{USCCON}(\mathbb{R}^m)$ implies the condition that $\{u_n\}$
 satisfies
the
 connectedness condition, however, the converse is false.
In fact, from
 clause (\romannumeral1) of Corollary \ref{cdpu}, it can be seen that we have a wide choice
of sequences in $F_{USC} (\mathbb{R}^m)$
which
 satisfy
the
 connectedness condition, but none of their members belongs to $F_{USCCON} (\mathbb{R}^m)$.
See
Examples \ref{wcre} and \ref{wcrepu}.
}
\end{re}

Now we arrive at the main results of this paper.

Compared to
 the corresponding result in \cite{huang} (Theorem \ref{com} in this paper),
the following Theorem \ref{fceg}
 greatly relaxes the prerequisite required for the compatibility of the endograph metric $H_{\rm end}$ and the $\Gamma$-convergence.

In Theorem \ref{com}, the sequence $\{u_n\}$ is required to be in $F_{USCGCON} (\mathbb{R}^m)$.
The following Theorem \ref{fceg}
shows that
not only
 the condition of the compactness of the positive $\al$-cuts
of the sequence
$\{u_n\}$ is not needed,
but also
 the condition of the connectedness
of
  all the positive $\al$-cuts
of the sequence
$\{u_n\}$ can be replaced by the condition of $\{u_n\}$ satisfies the connectedness condition.

\begin{tm} \label{fceg}
Let $u \in F_{USCG} (\mathbb{R}^m) \setminus \{\emptyset_{F(\mathbb{R}^m)}\} $,
and let
 $\{u_n, n=1,2,\ldots\}$ be a fuzzy set sequence in $F_{USC} (\mathbb{R}^m)$ which satisfies the connectedness condition.
Then $H_{\rm end} (u_n, u) \to 0$ as $n \to \infty$ if and only if $\lim_{n\to\infty}^{(\Gamma)} u_n = u$.
\end{tm}

\begin{proof}
   From Theorem \ref{hkg}, $H_{\rm end} (u_n, u) \to 0   \Rightarrow \lim_{n\to\infty}^{(\Gamma)} u_n = u$.

Now we show that $\lim_{n\to\infty}^{(\Gamma)} u_n = u \Rightarrow H_{\rm end} (u_n, u) \to 0$.
We prove by contradiction.
Assume that $\lim_{n\to\infty}^{(\Gamma)} u_n = u$ but $H_{\rm end} (u_n, u) \not\to 0$.
Then
$H^*({\rm end}\, u, {\rm end}\, u_n) \not\to 0$ or $H^*( {\rm end}\, u_n, {\rm end}\, u) \not\to 0$.
We split the proof into two cases.

\emph{Case} (\romannumeral1) $H^*({\rm end}\, u, {\rm end}\, u_n) \not\to 0$.

In this case, there is an $\varepsilon>0$ and a subsequence $\{u_{n_k}\}$ of $\{u_n\}$
such that
$H^*({\rm end}\, u, {\rm end}\, u_{n_k}) > \varepsilon$.
Thus for each $k=1,2,\ldots$ there exists
$(x_k, \al_k) \in {\rm end}\, u$ such that $\overline{d_m}((x_k, \al_k), {\rm end}\, u_{n_k}) > \varepsilon$.
Note that for each $k=1,2,\ldots$
$\al_k > \varepsilon$. So $\{(x_k, \al_k), \ k=1,2,\ldots\} \subseteq   {\rm end}_\varepsilon\, u$.

Since ${\rm end}_\varepsilon\, u $ is compact,
there is a subsequence $\{(x_{k_l}, \al_{k_l})\}$ of $\{(x_k, \al_k)\}$ which converges to $(x, \al) \in {\rm end}_\varepsilon\, u \subset {\rm end}\, u$.
Then
there is a $L(\varepsilon)$ such that
$\overline{d_m}((x, \al), (x_{k_l}, \al_{k_l})) < \varepsilon/2$ for all $l\geq L$. Hence,
by \eqref{uen},
$\overline{d_m}((x, \al), {\rm end}\, u_{n_{k_l}}) \geq \overline{d_m}((x_{k_l}, \al_{k_l}), {\rm end}\, u_{n_{k_l}}) - \overline{d_m}((x, \al), (x_{k_l}, \al_{k_l})) > \varepsilon/2$. By Lemma \ref{infe} (\romannumeral1), this contradicts
$(x, \al) \in {\rm end}\, u = \liminf_{n\to\infty} {\rm end}\, u_n$.

\emph{Case} (\romannumeral2) $H^*( {\rm end}\, u_n, {\rm end}\, u) \not\to 0$.

In this case,
 there is an $\varepsilon>0$ and a subsequence $\{u_{n_k}\}$ of $\{u_n\}$
such that
\begin{equation}\label{gem}
H^*({\rm end}\, u_{n_k}, {\rm end}\,u) > \varepsilon.
\end{equation}

Take $y \in \mathbb{R}^m$ with $u(y)>0$.
Set $\mu=u(y)$.
Since
$\{u_n\}$
 satisfies the connectedness condition, there is a $\xi\in (0, \min\{\mu, \varepsilon\})$
and $N_1 \in \mathbb{N}$
such that
${\rm end}_\xi \,u_n$ is connected in $\mathbb{R}^m \times [0,1]$
for all $n \geq N_1$.

Firstly we show the conclusions in the following clauses (\uppercase\expandafter{\romannumeral1}), (\uppercase\expandafter{\romannumeral2}) and (\uppercase\expandafter{\romannumeral3}).
\begin{description}
  \item[(\uppercase\expandafter{\romannumeral1})]
For each $n_k$, $k=1,2,\ldots$, there exists
$(x_{n_k}, \al_{n_k}) \in {\rm end}_\xi\, u_{n_k}$ with
$\overline{d_m}((x_{n_k}, \al_{n_k}), {\rm end}_\xi \, u) > \xi$.

  \item[(\uppercase\expandafter{\romannumeral2})]
There is an $N_2 \in \mathbb{N}$ such that for each $n\geq N_2$, there exists $(y_n, \beta_n) \in {\rm end}_\xi \,u_n$
with
$\overline{d_m}((y_n, \beta_n), {\rm end}_\xi\,u ) < \xi $.

  \item[(\uppercase\expandafter{\romannumeral3})]
Set $N_3 := \max\{N_1, N_2\}$.
For each $n_k \geq N_3$,
 there exists $(z_{n_k}, \gamma_{n_k}) \in {\rm end}_\xi \, u_{n_k}$ such that
\begin{equation}\label{eung}
\overline{d_m}((z_{n_k}, \gamma_{n_k}), {\rm end}_\xi\, u) = \xi.
\end{equation}

\end{description}

To show (\uppercase\expandafter{\romannumeral1}), let $k\in \mathbb{N}$.
From \eqref{gem},
 there exists
$(x_{n_k}, \al_{n_k}) \in {\rm end}\, u_{n_k}$ such that
$$
\overline{d_m}((x_{n_k}, \al_{n_k}), {\rm end}_\xi \, u) \geq \overline{d_m}((x_{n_k}, \al_{n_k}), {\rm end}\, u) > \varepsilon > \xi.
$$
Clearly
$\al_{n_k} > \varepsilon$, and then
$(x_{n_k}, \al_{n_k}) \in {\rm end}_\varepsilon \, u_{n_k} \subset  {\rm end}_\xi \, u_{n_k} $.
Thus (\uppercase\expandafter{\romannumeral1}) is true.

As
 $(y, \mu) \in {\rm end}\,u$ and ${\rm end}\,u = \liminf_{n\to\infty} {\rm end}\, u_n$,
 we can find a sequence $\{(y_n, \beta_n)\}$
satisfying that
$(y_n, \beta_n) \in {\rm end}\, u_n$ for
$n=1,2,\ldots$
and
$\{(y_n, \beta_n)\}$ converges to $(y, \mu)$.

Hence
 there is a $N_2$ such that for all $n\geq N_2$,
$\overline{d_m}((y_n, \beta_n), (y, \mu)) < \min\{\mu-\xi, \, \xi\}$.
Let $n\in \mathbb{N}$ with $n\geq N_2$.
Then
$\beta_n > \xi$,
and therefore $(y_n, \beta_n) \in {\rm end}_\xi \,u_n$.
Note that $(y, \mu) \in  {\rm end}_\xi\,u$. So $\overline{d_m}((y_n, \beta_n), {\rm end}_\xi \,u ) \leq \overline{d_m}((y_n, \beta_n), (y, \mu)) < \xi$.
Thus (\uppercase\expandafter{\romannumeral2}) is true.

To show (\uppercase\expandafter{\romannumeral3}),
let $k\in \mathbb{N}$ with $n_k \geq N_3$.
Define a function $f_k$ from $({\rm end}_\xi \,u_{n_k}, \overline{d_m})$ to $\mathbb{R}$ as follows:
$$f_k (z,\zeta) =\overline{d_m}((z,\zeta), {\rm end}_\xi\, u) \mbox{ for } (z,\zeta) \in {\rm end}_\xi \, u_{n_k}.$$

By \eqref{uen},
$|f_k (z,\zeta) - f_k (z',\zeta')| \leq \overline{d_m}((z,\zeta), (z',\zeta'))$ for $(z,\zeta), (z',\zeta')$ in ${\rm end}_\xi \, u_{n_k}$.
Thus
$f_k$  is a continuous function on ${\rm end}_\xi \, u_{n_k}$.

Note that
${\rm end}_\xi \, u_{n_k}$ is a connected set in $\mathbb{R}^m$.
Thus
$f_k ({\rm end}_\xi \, u_{n_k})$
 is a connected set in $\mathbb{R}$; that is, $f_k ({\rm end}_\xi \, u_{n_k})$ is an interval.
Combined this fact with the above
clauses (\uppercase\expandafter{\romannumeral1}) and (\uppercase\expandafter{\romannumeral2}), we obtain
that there exists $(z_{n_k}, \gamma_{n_k}) \in {\rm end}_\xi \, u_{n_k}$ with
$
\overline{d_m}((z_{n_k}, \gamma_{n_k}), {\rm end}_\xi\, u) = \xi.
$
Thus (\uppercase\expandafter{\romannumeral3}) is true.

Now using (\uppercase\expandafter{\romannumeral3}), we can obtain a contradiction.
From \eqref{eung} and the compactness of
${\rm end}_\xi\, u$, the set
 $\{(z_{n_k}, \gamma_{n_k}), n_k \geq N_3   \}$ is bounded in $\mathbb{R}^m \times [0,1]$, and thus $\{ (z_{n_k}, \gamma_{n_k}) , n_k \geq N_3   \}$ has a cluster point $(z, \gamma)$.
So
$(z,\gamma)\in \limsup_{n\to\infty} {\rm end} \, u_n = {\rm end}\, u$.
As $(z_{n_k}, \gamma_{n_k}) \in {\rm end}_\xi \, u_{n_k}$, we have that
$\gamma_{n_k} \geq \xi$ and therefore $\gamma\geq \xi$.
Thus
 $(z,\gamma)\in {\rm end} \, u \cap (X\times [\xi, 1]) = {\rm end}_\xi\, u$.
But, by \eqref{eung}, $\overline{d_m}((z,\gamma), {\rm end}_\xi\, u) = \xi$, which is a contradiction.

\end{proof}

\begin{tl} \label{comg} \
Let $u$ be a fuzzy set in $F_{USCG} (\mathbb{R}^m) \setminus \{\emptyset_{F(\mathbb{R}^m)}\}$
and for $n=1,2,\ldots$,
let $u_n$ be a fuzzy set in $F_{USCCON} (\mathbb{R}^m)$.
Then
$H_{\rm end} (u_n, u) \to 0$ as $n \to \infty$ if and only if $\lim_{n\to\infty}^{(\Gamma)} u_n = u$.
\end{tl}

\begin{proof}

Let $\{u_n: n\in \mathbb{N}\}$ is a sequence in $F_{USCCON} (\mathbb{R}^m)$. Then,
by clause (\romannumeral1) of Corollary \ref{cdpu},
for each $n\in \mathbb{N}$ and $r\in (0,1]$,
${\rm end}_r\, u_n$ is a connected set in $\mathbb{R}^m \times [0,1]$.
Hence $\{u_n: n\in \mathbb{N}\}$
 satisfies
the
 connectedness condition.
Thus the desired result follows immediately from
Theorem \ref{fceg}.

\end{proof}

Here we mention that
for $u \in F_{USCG}(\mathbb{R}^m)\setminus \{\emptyset_{F(\mathbb{R}^m)}\}$ and a sequence $\{u_n\}$
in $F_{USC}(\mathbb{R}^m)$ which satisfies the connectedness condition,
$H_{\rm end}(u_n, u) \to 0$ does not imply that
there is an $N$ such that
 for each $n\geq N$, $u_n \in F_{USCCON}(\mathbb{R}^m)$.
The following Examples \ref{wcre} and \ref{wcrepu} are two such examples.

Examples \ref{wcre} and \ref{wcrepu}
 show
that
there exists a
$u \in F^1_{USCG}(\mathbb{R})$ and a sequence $\{u_n\}$
in $F^1_{USC}(\mathbb{R})$
such that
\\
(\romannumeral1)
$H_{\rm end}(u_n, u) \to 0$,
\\
(\romannumeral2) $\{u_n\}$
satisfies the connectedness condition, and
\\
(\romannumeral3) for each $n=1,2,\ldots$, $u_n\notin F_{USCCON}(\mathbb{R})$.

\begin{eap} \label{wcre}
{\rm

For $n=1,2,\ldots$,
let
$u_n$ be the fuzzy set $u$ given in Example \ref{ucn};
that is,
 $u_n=u$ is a fuzzy set in $F^1_{USCG}(\mathbb{R})$ defined by putting:
\[
[u]_\al=
\left\{
  \begin{array}{ll}
   [0,1] \cup [3,4], & \al \in (0.6, 1],\\
  \mbox{} [0,4], &  \al\in [0,0.6].
  \end{array}
\right.
\]
We have the following conclusions:
\\
(\romannumeral1)
$H_{\rm end}(u_n, u) \to 0$ since $u_n = u$ for $n=1,2,\ldots$;
\\
(\romannumeral2) $\{u_n\}$
satisfies the connectedness condition because ${\rm end}_{r}\, u$ is connected in $\mathbb{R} \times [0,1]$ when $r\in  [0,0.6]$ (see Example \ref{ucn});
\\
(\romannumeral3)
for each $n=1,2,\ldots$, $u_n\notin F_{USCCON}(\mathbb{R})$, as $u\notin   F_{USCCON}(\mathbb{R})$ (see Example \ref{ucn}).

}
\end{eap}

\begin{eap} \label{wcrepu}
{\rm

For $n=1,2,\ldots$,
let
$u_n$ be the fuzzy set in $F^1_{USC}(\mathbb{R})$ defined by putting:
\[
[u_n]_\al=
\left\{
  \begin{array}{ll}
   [0,1] \cup [3,4], & \al \in (0.6, 1],\\
   \mbox{}[0,4], & \al \in (\frac{1}{3n}, 0.6], \\
  \mbox{} [0,9n] \cup [10n, +\infty), &  \al\in [0,\frac{1}{3n}].
  \end{array}
\right.
\]
 Let $u \in F^1_{USCG}(\mathbb{R})$ be defined by putting
\[
[u]_\al=
\left\{
  \begin{array}{ll}
   [0,1] \cup [3,4], & \al \in (0.6, 1],\\
  \mbox{} [0,4], &  \al\in [0,0.6].
  \end{array}
\right.
\]
We have the following conclusions:
\\
(\romannumeral1)
$H_{\rm end}(u_n, u) \to 0$ since $H_{\rm end}(u_n, u) = \frac{1}{3n}$ for $n=1,2,\ldots$.
\\
(\romannumeral2) $\{u_n\}$
satisfies the connectedness condition because
for each $\varepsilon>0$, if $\frac{1}{3n} < \varepsilon$, i.e. $n> \frac{1}{3\varepsilon}$, then
${\rm end}_{\varepsilon}\, u_n$ is connected in $\mathbb{R} \times [0,1]$.
\\
(\romannumeral3)
For each $n=1,2,\ldots$, $u_n\notin F_{USCG}(\mathbb{R})$.
 For each $n=1,2,\ldots$, $u_n\notin   F_{USCCON}(\mathbb{R})$,
further, ${\rm end}_{\alpha}\, u_n$ is not connected in $\mathbb{R}$ when $\al\in [0,\frac{1}{3n}]$.

}
\end{eap}

\begin{re}
  {\rm

Theorem 9.2 in \cite{huang}, which is Theorem \ref{com} in this paper,
is a corollary of
Corollary \ref{comg} since $F_{USCGCON} (\mathbb{R}^m) \subseteq F_{USCCON} (\mathbb{R}^m) $.

Theorem \ref{fce} is a corollary of Corollary \ref{comg}, as
$F^1_{USCG} (\mathbb{R}^m) \subseteq F_{USCG} (\mathbb{R}^m) \setminus \{\emptyset_{F(\mathbb{R}^m)}\}$
and
 $F^1_{USCCON} (\mathbb{R}^m) \subseteq F_{USCCON} (\mathbb{R}^m) $.

 Corollary \ref{comg} is a corollary of Theorem \ref{fceg}.

}
\end{re}

There exist a large variety
of convergent sequences in $F_{USC} (\mathbb{R}^m)$ with their limit fuzzy sets in
$F_{USCG}(\mathbb{R}^m)\setminus \{\emptyset_{F(\mathbb{R}^m)}\}$.
These sequences
 satisfy
the
 connectedness condition, but their members belong to neither $F_{USCCON} (\mathbb{R}^m)$ nor $F_{USCG} (\mathbb{R}^m)$.
A such example is given in Example \ref{wcrepu}.
So compared with Theorem 9.2 in \cite{huang} and Theorem \ref{fce},
Theorem \ref{fceg}
significantly relaxes the prerequisite required for the compatibility of the endograph metric $H_{\rm end}$ and the $\Gamma$-convergence.
 Theorem \ref{fceg}
improves Theorem 9.2 in \cite{huang} and Theorem \ref{fce}.

\section{Some supplementary discussions}

In this section, we discuss the properties of
the fuzzy set sequences with the connectedness condition
as a supplement.
Using
the result of this section, we give an improvement of the result in Section \ref{nce}.

First, we claim that
for $D, E \in C(X \times [0,1])$,
\begin{equation}\label{srhenu}
  H (D, E) \geq |S_D - S_E|.
\end{equation}
To see this,
let $D,E \in C(X \times    [0,1])$.
If $|S_D - S_E|=0$, then  \eqref{srhenu} is true. If $|S_D - S_E|>0$. Assume that $S_D > S_E$.
Note that for each $(x, t) \in D$ with $t> S_E$,
$\overline{d} ((x, t) , E) \geq t - S_E$. Thus
$H (D, E) \geq \sup   \{     t - S_E : (x, t) \in D \mbox{ with } t> S_E \}= S_D - S_E$.
So \eqref{srhenu} is true.

Let $u\in F(X)$.
Define $\bm{S_u} :=  \sup\{a:  u(x) =a\}$.
Clearly $[u]_{S_u} = \emptyset$ is possible.
We can see that $S_u = S_{{\rm end}\, u}$.

From \eqref{srhenu}, we have that
for $u,v \in F_{USC} (X)$,
\begin{equation}\label{srhe}
  H_{\rm end} (u, v) \geq |S_u - S_v|.
\end{equation}

\begin{pp}\label{sem}
  Let $u$ and $u_n$, $n=1,2,\ldots$, be fuzzy sets in $F_{USC}(X)$.
If $H_{\rm end} (u_n, u) \to 0$ as $n\to \infty$,
 then $S_{u_n} \to S_u$ as $n\to \infty$.
\end{pp}

\begin{proof}
The desired result follows immediately from \eqref{srhe}.

\end{proof}
We claim that:
(\uppercase\expandafter{\romannumeral1})
Let $u\in F_{USC}(X)$.
If there is an $\al\in [0, S_u]$ with $[u]_\al \in K(X)$, then $S_u =  \max\{a:  u(x) =a\}$.

If $\al = S_u$, then $[u]_{S_u} \not= \emptyset$ and the desired conclusion holds.
If $\al< S_u$, then pick a sequence $\{x_n\}$ in $[u]_\al$ with $u(x_n) \to S_u$.
From the compactness of $[u]_\al$, there is a subsequence
$\{x_{n_k}\}$ of $\{x_n\}$ such that
$\{x_{n_k}\}$ converges to a point $x$ in $[u]_\al$.
Thus
$u(x) \geq \lim_{k\to\infty} u(x_{n_k}) = S_u$.
Hence
$u(x) = S_u$ and the desired conclusion holds.
In the $\al< S_u$ case, we can also prove the desired conclusion as follows. Take an increasing sequence $\{\al_k\}$
in $[\al, 1]$ with $\al_k\to S_u-$.
Then $[u]_{\al_k} \in K(X)$ for each $k=1,2,\ldots$, and thus
$[u]_{S_u} = \cap_{k=1}^{+\infty} [u]_{\al_k} \in K(X)$.
So $[u]_{S_u} \not= \emptyset$ and
the desired conclusion holds.

If $u \in F_{USCG} (X) \setminus \{\emptyset_{F(X)}\}$, then from the above conclusion, $S_u =  \max\{a:  u(x) =a\}$.
If $u = \emptyset_{F(X)}$, then
$S_u = 0= \max\{a:  u(x) =a\}$.
So for each $u \in F_{USCG} (X)$,
$S_u =  \max\{a:  u(x) =a\}$.

Let $a\in [0,1]$. Define
\begin{gather*}
 F^{'a}_{USC} (X) = \{ u\in F_{USC} (X):  S_u = a\},
\\
 F^a_{USC} (X) = \{ u\in F_{USC} (X):  S_u = a,\ [u]_a\not=\emptyset\},
\\
  F^a_{USCG} (X) = \{ u\in F_{USCG}(X):  S_u = a\}.
\end{gather*}

\begin{tm}\label{csu}
Let $u \in F_{USCG} (\mathbb{R}^m) \setminus \{\emptyset_{F(\mathbb{R}^m)}\} $,
and let
 $\{u_n, n=1,2,\ldots\}$ be a fuzzy set sequence in $F_{USC} (\mathbb{R}^m)$ which satisfies the connectedness condition.
\\
(\romannumeral1) \ $\lim_{n\to\infty}^{(\Gamma)} u_n = u$ if and only if $H_{\rm end} (u_n, u) \to 0$ as $n \to \infty$.
\\
(\romannumeral2) \ If $\lim_{n\to\infty}^{(\Gamma)} u_n = u$, then
$u\in F^a_{USCG} (\mathbb{R}^m)$, where
 $a=\lim_{n\to\infty} S_{u_n} $.
\end{tm}

\begin{proof}
(\romannumeral1) is Theorem \ref{fceg}.
(\romannumeral2) follows immediately from (\romannumeral1) and Proposition \ref{sem}.

\end{proof}

The following
Corollaries \ref{fcegrua} and \ref{fcegru} follows immediately from Theorem \ref{csu}.
Corollary \ref{fcegru} is an immediate
consequence of Corollary \ref{fcegrua}, and
Theorem \ref{fce} is an immediate consequence of Corollary \ref{fcegru}.

\begin{tl} \label{fcegrua}
Let $u \in F_{USCG} (\mathbb{R}^m) \setminus \{\emptyset_{F(\mathbb{R}^m)}\} $
and let
 $\{u_n, n=1,2,\ldots\}$ be a fuzzy set sequence in $F_{USC} (\mathbb{R}^m)$ which satisfies the connectedness condition.
\\
(\romannumeral1) \ $\lim_{n\to\infty}^{(\Gamma)} u_n = u$ if and only if $H_{\rm end} (u_n, u) \to 0$ as $n \to \infty$.
\\
(\romannumeral2) \ If $\lim_{n\to\infty}^{(\Gamma)} u_n = u$ and $S_{u_n} \to 1$ as $n\to\infty$, then
$u \in F^1_{USCG} (\mathbb{R}^m) $.
\end{tl}

\begin{tl} \label{fcegru}
Let $u \in F_{USCG} (\mathbb{R}^m) \setminus \{\emptyset_{F(\mathbb{R}^m)}\} $
and let
 $\{u_n, n=1,2,\ldots\}$ be a fuzzy set sequence in $F^{'1}_{USC} (\mathbb{R}^m)$ which satisfies the connectedness condition.
\\
(\romannumeral1) \ $\lim_{n\to\infty}^{(\Gamma)} u_n = u$ if and only if $H_{\rm end} (u_n, u) \to 0$ as $n \to \infty$.
\\
(\romannumeral2) \ If $\lim_{n\to\infty}^{(\Gamma)} u_n = u$, then
$u \in F^1_{USCG} (\mathbb{R}^m) $.
\end{tl}

Let $u \in F_{USCG} (\mathbb{R}^m) \setminus \{\emptyset_{F(\mathbb{R}^m)}\} $,
and let
 $\{u_n, n=1,2,\ldots\}$ be a fuzzy set sequence in $F_{USC} (\mathbb{R}^m)$.
$\lim_{n\to\infty}^{(\Gamma)} u_n = u$ does not necessarily imply that $S_{u_n} \to S_u$ as $n\to \infty$.
Counterexamples are
given in the following Examples \ref{gen} and \ref{genr}.

\begin{eap} \label{gen}
  {\rm

Let $u\in F_{USCG} (\mathbb{R}) \setminus \{\emptyset_{F(\mathbb{R})}\} $ be defined by putting
\[
 [u]_\al = \left\{
             \begin{array}{ll}
           \emptyset, & \al\in (0.6, 1], \\
     \mbox{} [0,1], &   \al\in [0, 0.6].
             \end{array}
           \right.
\]
For $n=1,2,\ldots$, let $u_n \in F_{USC} (\mathbb{R})$ be defined by putting
\[
 [u_n]_\al = \left\{
             \begin{array}{ll}
          \{n\}, & \al\in (0.6, 1], \\
     \mbox{}  \{n\} \cup [0,1], &   \al\in [0, 0.6].
             \end{array}
           \right.
\]
Then clearly $\lim_{n\to\infty}^{(\Gamma)} u_n = u$.
However  $S_u= 0.6$
and
 $S_{u_n}= 1$ for $n=1,2,\ldots$,
and
hence
 $\lim_{n\to \infty} S_{u_n} =1 \not= S_u$.

}
\end{eap}

\begin{eap}   \label{genr}
  {\rm

Let $\{v_n\}$ be a sequence in $F_{USC} (\mathbb{R})$ defined by
\[
v_n=\left\{
  \begin{array}{ll}
    u_{n/2}, & n \mbox{ is even}, \\
  u, &  n \mbox{ is odd},
  \end{array}
\right.
\]
where $u$, $u_n$, $n=1,2,\ldots$ are defined in Example \ref{gen}.

Then clearly $\lim_{n\to\infty}^{(\Gamma)} v_n = u$.
However
\[
S_{v_n}=
\left\{
  \begin{array}{ll}
1, & n \mbox{ is even}, \\
  0.6, &  n \mbox{ is odd},
  \end{array}
\right.
\]
and hence
the sequence
  $\{S_{v_n}\}$
is not convergent.

}
\end{eap}

\begin{pp} \label{emu} Let $u$ in $F_{USC}(X)$
and
  let $\{u_n\}$ be a sequence in $F_{USC}(X)$ which converges to $u$ in $(F_{USC}(X), H_{\rm end})$.
 Let $\al\in (0,1]$ and $\beta\in (\alpha, 1]$. If $[u]_\al$ is bounded in $X$, then there is an $N\in \mathbb{N}$
 such that
 $\cup_{n \geq N} [u_n]_\beta$ is bounded in $X$.
\end{pp}

\begin{proof}
 Since $H_{\rm end}(u_n, u) \to 0$, there is an $N \in \mathbb{N}$
such that
\begin{equation*}
H_{\rm end} (u_{n},  u)  <  \beta-\al
\end{equation*}
for all $n \geq N$.
Thus if $n \geq N$ and $[u_n]_\beta \not= \emptyset$, then $[u]_\al \not= \emptyset$ and
\begin{equation}\label{cutsgmr}
H^*([u_n]_{\beta},       [u]_{\al})  <  \beta-\al.
\end{equation}

If $ \cup_{n \geq N} [u_n]_{\beta}\not=\emptyset$, then by \eqref{cutsgmr},
$
H^*( \cup_{n \geq N} [u_n]_{\beta},       [u]_{\al})  \leq \beta-\al.
$ Indeed, $
H^*( \cup_{n \geq N} [u_n]_{\beta},       [u]_{\al})  < \beta-\al.
$
Thus $ \cup_{n \geq N} [u_n]_{\beta}$ is bounded in $X$.

If $ \cup_{n \geq N} [u_n]_{\beta}=\emptyset$, then obviously $\cup_{n \geq N} [u_n]_{\beta}$ is bounded in $X$.

The proof is completed.

We can see that
if $[u]_\al = \emptyset$, then
 $[u_n]_\beta=\emptyset$ for all $n\geq N$, i.e.
 $\cup_{n \geq N} [u_n]_\beta = \emptyset$, where $N$ is the $N$ given in the first line of this proof.

\end{proof}

\begin{tl} \label{emun} Let $u \in F_{USC}(\mathbb{R}^m)$
and
  let $\{u_n\}$ be a sequence in $F_{USC}(\mathbb{R}^m)$ which converges to $u$ in $(F_{USC}(\mathbb{R}^m), H_{\rm end})$.
 Let $\al\in (0,1]$ and $\beta\in (\alpha, 1]$. If $[u]_\al$ is compact in $\mathbb{R}^m$, then there is an $N\in \mathbb{N}$
 such that for all $n\geq N$,
 $[u_n]_\beta$ is compact in $\mathbb{R}^m$.
\end{tl}

\begin{proof} Note that $[u_n]_\beta$ is compact in $\mathbb{R}^m$ means that $[u_n]_\beta$ is bounded in $\mathbb{R}^m$.
  So
  the desired result follows immediately from Proposition \ref{emu}.

\end{proof}

\begin{tl} \label{emungr} Let $u\in F_{USCG} (\mathbb{R}^m)$
and
  let $\{u_n\}$ be a sequence in $F_{USC}(\mathbb{R}^m)$ which converges to $u$ in $(F_{USC}(\mathbb{R}^m), H_{\rm end})$.
Then for each $\al\in (0,1]$, there is an $N(\al)$ in $\mathbb{N}$
 such that for all $n\geq N$,
 $[u_n]_\al$ is compact in $\mathbb{R}^m$.
\end{tl}

\begin{proof}
  Note that for each $\al\in (0,1]$, $[u]_{\al/2}$ is compact in $\mathbb{R}^m$.
  So
  the desired result follows immediately from Corollary \ref{emun}.

\end{proof}

\begin{pp} \label{emungrc} Let $u\in F_{USC} (\mathbb{R}^m) \setminus \{\emptyset_{F(\mathbb{R}^m)}\} $
and
  let $\{u_n\}$ be a sequence in $F_{USC}(\mathbb{R}^m)$ which converges to $u$ in $(F_{USC}(\mathbb{R}^m), H_{\rm end})$.
  Let $\al\in (0,S_u)$. If $[u]_\al$ is compact in $\mathbb{R}^m$,
then there is an $N$ in $\mathbb{N}$
 such that for all $n\geq N$,
 $u_n \in F^{S_{u_n}}_{USC} (\mathbb{R}^m)$.
\end{pp}

\begin{proof}
  Pick $\beta \in (\al, S_u)$.
 By Proposition \ref{sem}, there is an $N_1\in \mathbb{N}$
 such that for all $n\geq N_1$, $S_{u_n} > \beta$.
By Corollary \ref{emun}, there is an $N_2\in \mathbb{N}$
 such that for all $n\geq N_2$,
  $[u_n]_\beta$
is compact in $\mathbb{R}^m$.
Put $N_3 = \max\{N_1, N_2\}$.
Then for all $n\geq N_3$,
 $[u_n]_\beta \in K(\mathbb{R}^m)$.
So by affirmation (\uppercase\expandafter{\romannumeral1}) in Page 24,
 $u_n \in F^{S_{u_n}}_{USC} (\mathbb{R}^m)$.

\end{proof}

\section{Discussions}

Let $u\in F(X)$. Define $\bm{S_u} := \sup\{ u(x): x\in X\}$. Clearly $[u]_{S_u} = \emptyset$ is possible.

Let $\{u_n\}$ be a fuzzy set sequence in $F_{USC} (\mathbb{R}^m)$
and $\{u_{n_k}\}$ a subsequence of $\{u_n\}$.
If there is a $u \in  F_{USC} (\mathbb{R}^m) \setminus \{\emptyset_{F(\mathbb{R}^m)}\} $
 such that $u= \lim_{n\to\infty}^{(\Gamma)} u_n$,
then we can define
\begin{gather*}
  A_{n_k}^1 := \{\xi >0: \mbox{ for each } n_k, \ H^*({\rm end}\, u, {\rm end}\, u_{n_k}) > \xi \}, \\
    A_{n_k}^2 := \{\xi >0: \mbox{ for each } n_k, \ H^*({\rm end}\, u_{n_k}, {\rm end}\, u) > \xi \}.
\end{gather*}

 Let $\{u_n\}$ be a fuzzy set sequence in $F_{USC} (\mathbb{R}^m)$ and let $u$ be a fuzzy set in $F_{USC}(\mathbb{R}^m) \setminus \{\emptyset_{F(\mathbb{R}^m)}\}$.
We call the pair $\{u_n\}$, $u$
is a \textbf{\emph{weak connectedness compact pair}}
if one of the following (\romannumeral1) and (\romannumeral2) holds:
\\
(\romannumeral1) \ $\lim_{n\to\infty}^{(\Gamma)} u_n$ does not exist, or $\lim_{n\to\infty}^{(\Gamma)} u_n$ exists but $u \not= \lim_{n\to\infty}^{(\Gamma)} u_n$;
\\
(\romannumeral2) \  $\lim_{n\to\infty}^{(\Gamma)} u_n$ exists and $u= \lim_{n\to\infty}^{(\Gamma)} u_n$, moreover (\romannumeral2-1) and (\romannumeral2-2) listed below are true.

(\romannumeral2-1) \ for each
$\{u_{n_k}\}$
with $A_{n_k}^1 \not= \emptyset$,
 there exists a $\xi \in A_{n_k}^1$
such that ${\rm end}_\xi\, u$ is compact in $\mathbb{R}^m \times [0,1]$.

(\romannumeral2-2) \
for each
$\{u_{n_k}\}$
with $A_{n_k}^2 \not= \emptyset$,
 there exists a $\xi \in A_{n_k}^2$ and an $N(\xi) \in \mathbb{N}$
such that $\xi < S_u$, ${\rm end}_\xi\, u$ is compact in $\mathbb{R}^m \times [0,1]$, and
$ {\rm end}_{\xi}\, u_{n_k}$ is connected in $\mathbb{R}^m \times [0,1]$ for all $n_k\geq N$.

\begin{tm}
 \label{fcegun}
Let $u \in F_{USC} (\mathbb{R}^m) \setminus \{\emptyset_{F(\mathbb{R}^m)}\} $,
and let
 $\{u_n, n=1,2,\ldots\}$ be a fuzzy set sequence in $F_{USC} (\mathbb{R}^m)$.
 If the pair $\{u_n\}$, $u$
is a weak connectedness compact pair,
then $H_{\rm end} (u_n, u) \to 0$ as $n \to \infty$ if and only if $\lim_{n\to\infty}^{(\Gamma)} u_n = u$.

\end{tm}

\begin{proof}
  The proof is similar to that of Theorem \ref{fceg}.

  From Theorem \ref{hkg}, $H_{\rm end} (u_n, u) \to 0   \Rightarrow \lim_{n\to\infty}^{(\Gamma)} u_n = u$.

Now we show that $\lim_{n\to\infty}^{(\Gamma)} u_n = u \Rightarrow H_{\rm end} (u_n, u) \to 0$.
We prove by contradiction.
Assume that $\lim_{n\to\infty}^{(\Gamma)} u_n = u$ but $H_{\rm end} (u_n, u) \not\to 0$.
Then
$H^*({\rm end}\, u, {\rm end}\, u_n) \not\to 0$ or $H^*( {\rm end}\, u_n, {\rm end}\, u) \not\to 0$.
We split the proof into two cases.

\emph{Case} (\romannumeral1) $H^*({\rm end}\, u, {\rm end}\, u_n) \not\to 0$.

In this case, $A_{n_k}^1 \not= \emptyset$.
Since the pair $\{u_n\}$, $u$
is a weak connectedness compact pair,
then there is a $\xi \in  A_{n_k}^1 $ with ${\rm end}_\xi\, u$ is compact.
So we can prove that there is a contradiction in a similar manner to that in
the case (\romannumeral1) of the proof of Theorem \ref{fceg}.

\emph{Case} (\romannumeral2) $H^*( {\rm end}\, u_n, {\rm end}\, u) \not\to 0$.

In this case,
 $ A_{n_k}^2 \not= \emptyset $.
Since
the pair $\{u_n\}$, $u$
is a weak connectedness compact pair, there is a
$\xi\in  A_{n_k}^2$ and $N_1 \in \mathbb{N}$
such that $\xi < S_u$, ${\rm end}_\xi\, u$ is compact, and
for all $n_k \geq N_1$, ${\rm end}_\xi \,u_{n_k}$ is connected in $\mathbb{R}^m \times [0,1]$.

Firstly we show the conclusions in the following clauses (\uppercase\expandafter{\romannumeral1}), (\uppercase\expandafter{\romannumeral2}) and (\uppercase\expandafter{\romannumeral3}).
\begin{description}
  \item[(\uppercase\expandafter{\romannumeral1})]
For each $n_k$, $k=1,2,\ldots$, there exists
$(x_{n_k}, \al_{n_k}) \in {\rm end}_\xi\, u_{n_k}$ with
$\overline{d_m}((x_{n_k}, \al_{n_k}), {\rm end}_\xi \, u) > \xi$.

  \item[(\uppercase\expandafter{\romannumeral2})]
There is an $N_2 \in \mathbb{N}$ such that for each $n\geq N_2$, there exists $(y_n, \beta_n) \in {\rm end}_\xi \,u_n$
with
$\overline{d_m}((y_n, \beta_n), {\rm end}_\xi\,u ) < \xi $.

  \item[(\uppercase\expandafter{\romannumeral3})]
Set $N_3 := \max\{N_1, N_2\}$.
For each $n_k \geq N_3$,
 there exists $(z_{n_k}, \gamma_{n_k}) \in {\rm end}_\xi \, u_{n_k}$ such that
\begin{equation}\label{eungcm}
\overline{d_m}((z_{n_k}, \gamma_{n_k}), {\rm end}_\xi\, u) = \xi.
\end{equation}

\end{description}

Note that $\xi\in  A_{n_k}^2$; that is,
 for each $u_{n_k}$, $k=1,2,\ldots$
\begin{equation}\label{gemu}
H^*({\rm end}\, u_{n_k}, {\rm end}\,u) > \xi.
\end{equation}
Let $k\in \mathbb{N}$.
From \eqref{gemu},
there exists
$(x_{n_k}, \al_{n_k}) \in {\rm end}\, u_{n_k}$ such that
$$
\overline{d_m}((x_{n_k}, \al_{n_k}), {\rm end}_\xi \, u) \geq \overline{d_m}((x_{n_k}, \al_{n_k}), {\rm end}\, u) > \xi.
$$
Clearly
$\al_{n_k} > \xi$, and then
$(x_{n_k}, \al_{n_k}) \in {\rm end}_\xi \, u_{n_k} $.
Thus (\uppercase\expandafter{\romannumeral1}) is true.

Since $\xi < S_u$, we can
take $y \in \mathbb{R}^m$ with $u(y)>\xi>0$. Set $u(y)=\mu$.
As $(y, \mu) \in {\rm end}\,u$ and ${\rm end}\,u = \liminf_{n\to\infty} {\rm end}\, u_n$, we can find a sequence $\{(y_n, \beta_n)\}$
satisfying that
$(y_n, \beta_n) \in {\rm end}\, u_n$ for
$n=1,2,\ldots$
and
$\{(y_n, \beta_n)\}$ converges to $(y, \mu)$.

Hence
 there is a $N_2$ such that for all $n\geq N_2$,
$\overline{d_m}((y_n, \beta_n), (y, \mu)) < \min\{\mu - \xi,\, \xi\}$.
Let $n\in \mathbb{N}$ with $n\geq N_2$.
Then $\beta_n > \xi$,
and therefore $(y_n, \beta_n) \in {\rm end}_\xi \,u_n$.
Note that $(y, \mu) \in  {\rm end}_\xi\,u$. So $\overline{d_m}((y_n, \beta_n), {\rm end}_\xi \,u ) \leq \overline{d_m}((y_n, \beta_n), (y, \mu)) < \xi$.
Thus (\uppercase\expandafter{\romannumeral2}) is true.

(\uppercase\expandafter{\romannumeral3}) follows from (\uppercase\expandafter{\romannumeral1}), (\uppercase\expandafter{\romannumeral2}) and the connectedness of ${\rm end}_\xi \,u_{n_k}$ when $n_k\geq N_3$. The proof of (\uppercase\expandafter{\romannumeral3})
is the same as that of
the clause
 (\uppercase\expandafter{\romannumeral3}) in the proof of Theorem \ref{fceg}.

Now using (\uppercase\expandafter{\romannumeral3}) and the compactness of ${\rm end}_\xi\, u$, we can have a contradiction.
The proof is the same as the counterpart in the proof of Theorem \ref{fceg}.

\end{proof}

\begin{re}\label{hwc}
  {\rm
Let $\{u_n\}$ be a fuzzy set sequence in $F_{USC} (\mathbb{R}^m)$ and let $u$ be a fuzzy set in $F_{USC}(\mathbb{R}^m) \setminus \{\emptyset_{F(\mathbb{R}^m)}\}$.
It is easy to see that
if $H_{\rm end} (u_n, u) \to 0$, then the pair $\{u_n\}$, $u$
is a weak connectedness compact pair.
}
\end{re}

Here we mention that
for $u \in F_{USC} (\mathbb{R}^m) \setminus \{\emptyset_{F(\mathbb{R}^m)}\} $,
and
fuzzy set sequence
 $\{u_n, n=1,2,\ldots\}$ in $F_{USC} (\mathbb{R}^m)$,
$H_{\rm end}(u_n, u) \to 0$ does not imply that
$u\in F_{USCG} (\mathbb{R}^m)$
or
$\{u_n\}$
satisfies the connectedness condition.

There exists
a sequence $\{u_n\}$
in $F_{USC}(\mathbb{R}^m)$ and an element $u$ in $F_{USC}(\mathbb{R}^m)\setminus \{\emptyset_{F(\mathbb{R}^m)}\}$
satisfying the following conditions (\romannumeral1)-(\romannumeral3).
\\
(\romannumeral1)
$H_{\rm end}(u_n, u) \to 0$, and then by Remark \ref{hwc}, the pair $\{u_n\}$, $u$
is a weak connectedness compact pair.
\\
(\romannumeral2) $u\notin F_{USCG} (\mathbb{R}^m)$.
\\
(\romannumeral3) $\{u_n\}$
does not satisfy the connectedness condition.

The following Examples \ref{wcregu} and \ref{wcreguf} are such examples.

\begin{eap} \label{wcregu}
{\rm

Let $u$ be a fuzzy set in $F^1_{USC}(\mathbb{R}) \setminus  \{\emptyset_{F(\mathbb{R})}\}$ defined by putting:
\[
[u]_\al=
\left\{
  \begin{array}{ll}
  \{1\}, & \al \in (0.6, 1],\\
  \mbox{} (-\infty,-1] \cup [1, +\infty) &  \al\in [0,0.6].
  \end{array}
\right.
\]
For $n=1,2,\ldots$,
let
$u_n=u$.
We have the following conclusions:
\\
(\romannumeral1)
$H_{\rm end}(u_n, u) \to 0$ since $u_n = u$ for $n=1,2,\ldots$. So the pair $\{u_n\}$, $u$
is a weak connectedness compact pair.
\\
(\romannumeral2) $u\notin F_{USCG} (\mathbb{R})$.
\\
(\romannumeral3) $\{u_n\}$
does not satisfy the connectedness condition because ${\rm end}_{r}\, u$ is not connected in $\mathbb{R} \times [0,1]$ when $r\in  [0,0.6]$. Moreover, it can be seen that each subsequence $\{u_{n_k}\}$ of $\{u_n\}$
does not satisfy the connectedness condition.

}
\end{eap}

\begin{eap} \label{wcreguf}
  {\rm

Let $u$ be a fuzzy set in $F^1_{USC}(\mathbb{R}) \setminus  \{\emptyset_{F(\mathbb{R})}\}$ defined by putting:
\[
[u]_\al=
\left\{
  \begin{array}{ll}
  \{1\}, & \al \in (0.6, 1],\\
  \mbox{} [1,3] \cup [4, +\infty) &  \al\in [0,0.6].
  \end{array}
\right.
\]
For $n=1,2,\ldots$,
let
$u_n$ be the fuzzy set in $F^1_{USC}(\mathbb{R})$ defined by putting
\[
[u_n]_\al=
\left\{
  \begin{array}{ll}
  \{1\}, & \al \in (0.6, 1],\\
  \mbox{} [1,3] \cup [4, +\infty) &  \al\in (\frac{1}{3n}, 0.6]\\
  \mbox{} [-\infty,-4] \cup [-3, -1] \cup [1,3] \cup [4, +\infty) &  \al\in [0, \frac{1}{3n}].
  \end{array}
\right.
\]
We have the following conclusions:
\\
(\romannumeral1)
$H_{\rm end}(u_n, u) \to 0$ since $H_{\rm end}(u_n, u) = \frac{1}{3n}$ for $n=1,2,\ldots$. So the pair $\{u_n\}$, $u$
is a weak connectedness compact pair.
\\
(\romannumeral2) $u\notin F_{USCG} (\mathbb{R})$.
\\
(\romannumeral3) $\{u_n\}$
does not satisfy the connectedness condition because for each $n=1,2,\ldots$,
${\rm end}_{r}\, u_n$ is not connected in $\mathbb{R} \times [0,1]$ when $r\in  [0,0.6]$. Indeed, it can be seen that each subsequence $\{u_{n_k}\}$ of $\{u_n\}$
does not satisfy the connectedness condition.

}
\end{eap}

\begin{re}
  {\rm

 Let $\{u_n\}$ be a fuzzy set sequence in $F_{USC} (\mathbb{R}^m)$ satisfying the connectedness condition and let $u$ be a fuzzy set in $F_{USCG}(\mathbb{R}^m) \setminus \{\emptyset_{F(\mathbb{R}^m)}\}$.
Then clearly
the pair $\{u_n\}$, $u$
is a weak connectedness compact pair.
So Theorem \ref{fceg} is a corollary of Theorem \ref{fcegun}.

}
\end{re}

\section{Some discussion}

\begin{tm} \cite{huang719}\label{tbe}
  Let $(X,d)$ be a metric space and $\mathcal{D}\subseteq K(X)$.
   Then $\mathcal{D}$ is totally bounded in $(K(X), H)$
if and only if
$\mathbf{ D} =   \bigcup \{C:  C \in  \mathcal{D} \} $ is totally bounded in $(X,d)$.
\end{tm}

\begin{tm} \cite{huang719} \label{rce}
  Let $(X,d)$ be a metric space and $\mathcal{D}\subseteq K(X)$. Then $\mathcal{D} $ is relatively compact in $(K(X), H)$
if and only if
$\mathbf{ D} =   \bigcup \{C:  C \in  \mathcal{D} \} $ is relatively compact in $(X,d)$.
\end{tm}

\begin{tm} \cite{huang719} \label{come}
 Let $(X,d)$ be a metric space and $\mathcal{D}\subseteq K(X)$. Then
    the following are equivalent:
    \\
    (\romannumeral1) \ $\mathcal{D} $ is compact in $(K(X), H)$;
        \\
 (\romannumeral2) \
$\mathbf{ D} =   \bigcup \{C:  C \in  \mathcal{D} \} $ is relatively compact in $(X, d)$
and
$\mathcal{D} $ is closed in $(K(X), H)$;
    \\
 (\romannumeral3) \
$\mathbf{ D} =   \bigcup \{C:  C \in  \mathcal{D} \} $ is compact in $(X, d)$
and
$\mathcal{D} $ is closed in $(K(X), H)$.
\end{tm}

Let $u\in F_{USC}(X)$.
Define $\bm{u^e}= {\rm end}\, u $.
Let $A$ be a subset of $F_{USC}(X)$.
Define
$\bm{A^e}= \{u^e:  u\in A\}$.
Clearly $F_{USC}(X)^e\subseteq C(X\times [0,1])$.

Define
$g: (F_{USC}(X), H_{\rm end}) \to (C(X\times [0,1]), H)$ given by
$g(u) = {\rm end}\, u$.
Then
\begin{itemize}
\item
$g$ is an isometric embedding
of
 $(F_{USC}(X), H_{\rm end}) $ in $(C(X\times [0,1]), H)$,

  \item $g(F_{USC}(X)) = F_{USC}(X)^e$, and

\item
$(F_{USC}(X), H_{\rm end}) $ is isometric
to
$(F_{USC}(X)^e, H)$.
\end{itemize}

The following representation theorem for $F_{USC}(X)^e$ follows immediately
from Proposition \ref{fus}.

\begin{pp} \label{repu}
Let $U$ be a subset of $X\times [0,1]$.
Then $U\in F_{USC}(X)^e$
if and only if the following properties (\romannumeral1)-(\romannumeral3) are true.
\\
(\romannumeral1) For each $\alpha \in (0, 1]$, $\langle U \rangle_\al \in C(X) \cup \{\emptyset\}$.
\\
(\romannumeral2) \
For each $\alpha \in (0, 1]$,
$\langle U\rangle_\al = \bigcap_{\beta<\al} \langle U\rangle_\beta $.
\\
(\romannumeral3) \ $\langle U\rangle_0 = X $.
\end{pp}

\begin{pp}
  \label{ucen}
Let $U\in C(X\times [0,1])$. Then the following (\romannumeral1) and (\romannumeral2)
are equivalent.
\\
(\romannumeral1) For each $\alpha$ with $0 < \al \leq 1$,
$\langle U\rangle_\al = \bigcap_{\beta<\al} \langle U\rangle_\beta $.
\\
(\romannumeral2)
For each $\alpha,\beta$ with $0\leq \beta < \al \leq 1$,
$\langle U\rangle_\al \subseteq \langle U\rangle_\beta $.
\end{pp}

\begin{proof} The proof is routine.
  (\romannumeral1)$\Rightarrow$(\romannumeral2) is obviously.

  Suppose that (\romannumeral2) is true.
 To show that (\romannumeral1) is true, let $\al\in (0,1]$.
      From (\romannumeral2), $\langle U\rangle_\al \subseteq \bigcap_{\beta<\al} \langle U\rangle_\beta $.
So we only need to prove
that
$\langle U\rangle_\al \supseteq \bigcap_{\beta<\al} \langle U\rangle_\beta $.
   To do this, let
   $x\in \bigcap_{\beta<\al} \langle U\rangle_\beta$.
   This means that $(x,\beta) \in U$ for $\beta\in [0,\al)$.
   Since $\lim_{\beta\to \al-} \overline{d}((x,\beta), (x, \al)) = 0$, from the closedness of
$U$, it follows that $(x, \al)\in U$.
Hence $x\in \langle U \rangle_\al$.
Thus $\langle U\rangle_\al \supseteq \bigcap_{\beta<\al} \langle U\rangle_\beta $ from the arbitrariness of $x$ in $\bigcap_{\beta<\al} \langle U\rangle_\beta$.
   So
 (\romannumeral2)$\Rightarrow$(\romannumeral1).

\end{proof}

\begin{pp} \label{repc}
Let $U\in C(X\times [0,1])$.
Then
$U\in F_{USC}(X)^e$
if and only if $U$ has the following properties:
\\
(\romannumeral1) \
for each $\alpha,\beta$ with $0\leq \beta < \al \leq 1$,
$\langle U\rangle_\al \subseteq \langle U\rangle_\beta $, and
\\
(\romannumeral2) \ $\langle U\rangle_0 = X$.
\end{pp}

\begin{proof}
Since
 $U\in C(X\times [0,1])$, then clearly $\langle U \rangle_\al \in C(X) \cup \{ \emptyset \}$
for all $\al\in [0,1]$.
Thus the desired result follows immediately from Propositions \ref{repu} and \ref{ucen}.

\end{proof}

As a shorthand, we denote the sequence $x_1, x_2, \ldots, x_n, \ldots$ by $\{x_n\}$.

\begin{pp}\label{usce}
$F_{USC}(X)^e$ is a closed subset of $(C(X \times [0,1]), H)$.
\end{pp}

\begin{proof}
  Let $\{u_n^e:   n=1,2,\ldots \}$ be a sequence in $F_{USC}(X)^e$
with
$\{u_n^e\}$ converging to $U$ in $(C(X \times [0,1]), H)$.
To show the desired result, we
only need to show
that $U \in F_{USC}(X)^e$.

We claim that
\\
(\romannumeral1) \
for each $\alpha,\beta$ with $0\leq \beta < \alpha \leq 1$,
$\langle U\rangle_\al \subseteq \langle U\rangle_\beta $;
\\
(\romannumeral2) \ $\langle U\rangle_0 = X $.

To show (\romannumeral1),
let
 $\alpha,\beta$ in $[0,1]$ with $\beta < \alpha$,
and let
 $x \in \langle U\rangle_\al$, i.e. $(x,\al) \in U$.
By Theorem \ref{hkg},
$\lim_{n\to \infty}^{(K)} u_n^e \, = U$.
Then there is a sequence $\{(x_n, \al_n)\}$ satisfying $(x_n, \al_n) \in u_n^e$ for $n=1,2,\ldots$
and
$\lim_{n\to\infty} \overline{d} ((x_n, \al_n), (x,\al)) = 0$.
Hence there is an $N$ such that $\al_n > \beta$ for all $n\geq N$.
Thus $(x_n, \beta) \in u_n^e$ for $n\geq N$. Note that
$\lim_{n\to\infty} \overline{d} ((x_n, \beta), (x,\beta)) = 0$.
Then $(x,\beta) \in \lim_{n\to \infty}^{(K)} u_n^e = U$.
This means that
$x\in \langle U\rangle_\beta$.
So (\romannumeral1) is true.

Clearly $\langle U \rangle_0 \subseteq X$.
From $\lim_{n\to \infty}^{(K)} u_n^e \, = U$
and $\langle u_n^e \rangle_0 = X$,
we have that
$\langle U \rangle_0 \supseteq X$.
Thus
$\langle U \rangle_0 = X$.
So (\romannumeral2) is true.

By Proposition \ref{repc}, (\romannumeral1) and (\romannumeral2) imply that
 $U \in F_{USC}(X)^e$.

\end{proof}

\begin{re}
  {\rm
Let $a\in [0,1]$.
From Proposition \ref{sem}, we  can deduce that
$F^{'a}_{USC}(X)$ is a closed subset of $(F_{USC}(X), H_{\rm end})$.
Then by Proposition \ref{usce}, we have that $F^{'a}_{USC}(X)^e$ is a closed subset of $(C(X \times [0,1]), H)$.
}
\end{re}

We use $(\widetilde{X}, \widetilde{d})$ to denote the completion of $(X, d)$.
We see $(X, d)$ as a subspace of $(\widetilde{X}, \widetilde{d})$.

If there is no confusion,
 we also
use $H$ to denote the Hausdorff metric
on
 $C(\widetilde{X})$ induced by $\widetilde{d}$.
  We also
use $H$ to denote the Hausdorff metric
 on $C(\widetilde{X}\times [0,1])$ induced by $\overline{\widetilde{d}}$.
 We
 also use $H_{\rm end}$ to denote the endograph metric on $F_{USC}(\widetilde{X})$
given by using $H$ on
$C(\widetilde{X} \times [0,1])$.

We see
$(F_{USCG} (X), H_{\rm end})$ as a metric subspace of $(F_{USCG} (\widetilde{X}), H_{\rm end})$. So a subset $U$
of $F_{USCG} (X)$ can be seen as a subset of $F_{USCG}(\widetilde{X})$.

Suppose that
$U$ is a subset of $F_{USC} (X)$ and $\al\in [0,1]$.
For
writing convenience,
we denote
\begin{itemize}
  \item  $U(\al):= \bigcup_{u\in U} [u]_\al$, and

\item  $U_\al : =  \{[u]_\al: u \in U\}$.
\end{itemize}

\begin{lm}
   \label{tbfegnum}
  Let $U$ be a subset of $F_{USCG} (X)$. If $U$ is totally bounded in $(F_{USCG} (X), H_{\rm end})$,
then
$U(\al)$
is totally bounded in $(X,d)$ for each $\al \in (0,1]$.
\end{lm}

\begin{proof} The proof is similar to that of the necessity part of Theorem 7.8 in \cite{huang719}.

Let $\al \in (0,1]$.
To show
that
$U(\al)$
is totally bounded in $X$, we only need to show that
each sequence in $U(\al)$
has a Cauchy subsequence.

Let $\{x_n\}$ be a sequence in $U(\al)$. Then there is a sequence $\{u_n\}$ in $U$ with $x_n \in  [u_n]_\al$ for $n=1,2,\ldots$.
Since $U$ is totally bounded in $(F_{USCG} (X), H_{\rm end})$,
$\{u_n\}$ has
 a Cauchy subsequence $\{u_{n_l}\}$ in $(F_{USCG} (X), H_{\rm end})$. So given $\varepsilon\in (0, \al)$, there is a $K( \varepsilon) \in \mathbb{N}$
such that
\begin{equation*}\label{ends}
H_{\rm end} (u_{n_l},  u_{n_K})  <   \varepsilon
\end{equation*}
for all $l \geq K$.
Thus
\begin{equation}\label{cutsgm}
H^*([u_{n_l}]_\al,       [u_{n_K}]_{\al-\varepsilon}) <  \varepsilon
\end{equation}
for all $l \geq K$.
From \eqref{cutsgm} and the arbitrariness of $\varepsilon$,  $\bigcup_{l=1} ^{+\infty} [u_{n_l}]_\al $ is totally bounded in $(X,d)$.
Thus
$\{x_{n_l}\}$, which is a subsequence of $\{x_n\}$,
has a Cauchy subsequence, and so does $\{x_n\}$.

\end{proof}

\begin{re}{\rm
 It is easy to see that for a totally bounded set $U$ in $(F_{USCG} (X), H_{\rm end})$ and $\al\in (0,1]$,
$U(\al) = \emptyset$ is possible even if $U \not= \emptyset$.
}
\end{re}

For $D \subseteq  X \times [0,1]$ and $\al\in [0,1]$,
define
 $\bm{\langle D \rangle_\al} := \{ x:   (x,\al) \in D\}$.

\begin{tm}
 \label{rcfegnum}
  Let $U$ be a subset of $F_{USCG} (X)$. Then $U$ is relatively compact in $(F_{USCG} (X), H_{\rm end})$
if and only if
$U(\al)$
is relatively compact in $(X, d)$ for each $\al \in (0,1]$.

\end{tm}

\begin{proof}

\textbf{\emph{Necessity}}. Suppose that $U$ is relatively compact in $(F_{USCG} (X), H_{\rm end})$.
Let $\al\in (0,1]$.
Then by Lemma \ref{tbfegnum},
 $U(\al)$ is totally bounded.
Hence
$U(\al)$ is relatively compact in $(\widetilde{X}, \widetilde{d})$.

To show $U(\al)$ is relatively compact in $(X,d)$, we proceed by  contradiction.
If this were not the case, then
there exists a sequence $\{x_n\}$ in $U(\al)$
such that $\{x_n\}$ converges to $x\in \widetilde{X} \setminus X$ in $(\widetilde{X}, \widetilde{d})$.

Assume that $x_n \in [u_n]_\al$ and $u_n \in U$, $n=1,2,\ldots$.
From the relative compactness of $U$,
there is a subsequence $\{u_{n_k}\}$ of $\{u_n\}$
such that
$\{u_{n_k}\}$
converges to $u \in F_{USCG}(X)$.
Since
$F_{USCG}(X)$ can be seen as a subspace of $F_{USCG}(\widetilde{X})$,
we obtain that $\{u_{n_k}\}$ converges to $u$ in $(F_{USCG}(\widetilde{X}), H_{\rm end})$.
By Theorem \ref{hkg},
$\lim_{k\to \infty}^{(K)} u_{n_k}^e \, = u^e$ according to $(\widetilde{X}\times[0,1], \overline{\widetilde{d}})$.
Notice
that
$(x_{n_k}, \al) \in u_{n_k}^e $ for $k=1,2,\ldots$, and
$\{(x_{n_k}, \al)\}$ converges to $(x, \al)$ in $(\widetilde{X}\times[0,1], \overline{\widetilde{d}})$.
Thus
$(x,\al) \in u^e$, which contradicts $x\in \widetilde{X} \setminus X$.

It can be seen that
the necessity part of Theorem 7.10 in \cite{huang719} can be verified in a similar manner to that in the necessity
part of this theorem.

\textbf{\emph{Sufficiency}}. \
Suppose that
$U(\al)$
is relatively compact in $(X, d)$ for each $\al \in (0,1]$.
To show that $U$ is relatively compact in $(F_{USCG} (X), H_{\rm end})$,
we only need to show
that each sequence in $U$ has a convergent subsequence in $(F_{USCG} (X), H_{\rm end})$.

Let $\{u_n\}$ be a sequence in $U$.
If
$\liminf_{n\to\infty} S_{u_n} = 0$, i.e. there is a subsequence $\{u_{n_k}\}$ of $\{u_n\}$
such that
$\lim_{k\to\infty} S_{u_{n_k}} = 0$. Then clearly $H_{\rm end} (u_{n_k}, \emptyset_{F(X)}) = S_{u_{n_k}} \to 0$
as $k \to \infty$.
Since $\emptyset_{F(X)} \in F_{USCG} (X)$, $\{u_{n_k}\}$ is a convergent subsequence in $(F_{USCG} (X), H_{\rm end})$.

If
$\liminf_{n\to\infty} S_{u_n} >0$, then there is a $\xi>0$ and an $N \in \mathbb{N}$
such that $[u_n]_{\xi} \not= \emptyset$
for all $n\geq N$.

First we claim the following property:
\begin{description}
  \item[(a)]
Let $\al\in (0,1]$ and $S$ be a subset of $U$ with $[u]_\al \not= \emptyset$ for each $u\in S$. Then
$\{ {\rm end}_\al u: u\in S \}$ is a relatively compact set in $(K(X\times [\al,1]), H)$.
\end{description}

It can be seen that for each $u\in S$,
${\rm end}_\al\, u \in K(X\times [\al,1])$.

As $U(\al)$ is relatively compact in $(X,d)$,
$U(\al) \times [\al, 1]$ is relatively compact in $(X\times [\al,1], \overline{d})$.
Since
$\bigcup_{u\in S} {\rm end}_\al u $ is
 a subset of $ U(\al) \times [\al, 1]$,
then
$\bigcup_{u\in S} {\rm end}_\al u $
is also a
relatively compact set in $(X\times [\al,1], \overline{d})$.
Thus
by Theorem \ref{rce}, $\{ {\rm end}_\al u: u\in S \}$ is relatively compact in $(K(X\times [\al,1]), H)$.
So affirmation (a) is true.

 Take a sequence $ \{\alpha_k, \ k=1,2,\ldots\}$ which satisfies that $0<\al_{k+1} <  \al_k \leq \min\{\xi, \frac{1}{k}\}$ for $k=1,2,\ldots$. We can see that $\alpha_k \to 0$ as $k\to\infty$.

By affirmation (a),
$\{ {\rm end}_{\al_1}\, u_n: n=N,N+1,\ldots\}$ is relatively compact in $(K(X\times [\al_1,1]), H)$.
So there is a subsequence $\{u_{n}^{(1)}\}$ of $\{u_n: n\geq N\}$ and $v^1 \in K(X\times [\al_1,1])$ such that $H({\rm end}_{\al_1}\, u_{n}^{(1)},  v^1 ) \to 0$.
Clearly, $\{u_{n}^{(1)}\}$ is also a subsequence of $\{u_n\}$.

Again using affirmation (a),
$\{ {\rm end}_{\al_2}\, u_n^{(1)} \}$ is relatively compact in $(K(X\times [\al_2,1]), H)$.
So there is a subsequence $\{u_{n}^{(2)}\}$ of $\{u_n^{(1)}\}$ and $v^2 \in K(X\times [\al_2,1])$ such that $H({\rm end}_{\al_2}\, u_{n}^{(2)},  v^2 ) \to 0$.

Repeating the above procedure, we can obtain $\{u_{n}^{(k)}\}$
and
$v^{k}\in K(X\times [\al_k, 1])$, $k=1,2,\ldots$,
such that
for each $k=1,2,\ldots$, $\{u_{n}^{(k+1)}\}$ is a subsequence of $\{u_{n}^{(k)}\}$
and
$H({\rm end}_{\al_k}\, u_{n}^{(k)},  v^k ) \to 0$.

We claim that
\begin{description}
\item[(b)] Let $k_1$ and $k_2$ be in $\mathbb{N}$
with $k_1 < k_2$. Then
\\
(\romannumeral1) \ $
\langle v^{k_1} \rangle_{\al_{k_1}} \subseteq \langle v^{k_2} \rangle_{\al_{k_1}}$,
\\
(\romannumeral2)
 $\langle v^{k_1} \rangle_\al = \langle v^{k_2} \rangle_\al$ when $\al\in (\al_{k_1}, 1]$,
 \\
(\romannumeral3) $v^k \subseteq v^{k+1}$ for $k=1,2,\ldots$.
 \end{description}

Note that $\{ u_n^{(k_2)}\}$ is a subsequence of $\{ u_n^{(k_1)}\}$
and that
$\al_{k_2} < \al_{k_1}$. Thus
by Theorem \ref{hkg}, for each $\al\in [\al_{k_1}, 1]$,
 \begin{align} \label{cner}
    &             \langle v^{k_1} \rangle_{\al}  \times \{\al\}   \nonumber \\
&=  \liminf_{n\to \infty} {\rm end}_{\al_{k_1}}\, u_n^{(k_1)} \cap (X \times \{\al\}) \nonumber\\
&
\subseteq \liminf_{n\to \infty} {\rm end}_{\al_{k_2}}\, u_n^{(k_2)} \cap (X \times \{\al\}) \nonumber\\
&
=\langle v^{k_2} \rangle_{\al}  \times \{\al\}.
                       \end{align}
So (\romannumeral1) is true.

Let $\al\in [0,1]$ with $\al > \al_{k_1}$. Observe that
if a sequence
$\{(x_m, \beta_m)\}$ converges to
a point $(x,\al)$ as $m\to\infty$ in $(X\times [0,1], \overline{d})$,
then
there is an $M$ such that for all $m>M$, $\beta_m >  \al_{k_1}$, i.e. $(x_m, \beta_m) \in X\times (\al_{k_1},1]$.
Thus by Theorem \ref{hkg}, for each $\al\in (\al_{k_1}, 1]$,
 \begin{align}  \label{cme}
    &             \langle v^{k_1} \rangle_\al  \times \{\al\} \nonumber \\
&=  \limsup_{n\to \infty} {\rm end}_{\al_{k_1}}\, u_n^{(k_1)}  \cap (X \times \{\al\})\nonumber \\
&
\supseteq \limsup_{n\to \infty} {\rm end}_{\al_{k_2}}\, u_n^{(k_2)}  \cap (X \times \{\al\}) \nonumber\\
&
=\langle v^{k_2} \rangle_\al  \times \{\al\}.
                       \end{align}

Hence by \eqref{cner} and \eqref{cme},
$ \langle v^{k_1} \rangle_\al =  \langle v^{k_2} \rangle_\al $ for $\al\in (\al_{k_1}, 1]$. So (\romannumeral2) is true.
(\romannumeral3)
follows immediately from (\romannumeral1) and (\romannumeral2).

Define a subset $v$ of $X \times [0,1]$
given by
\begin{equation}\label{fve}
v = \cup_{k=1}^{+\infty} v^k \cup ( X\times \{0\}).
\end{equation}
From affirmation (b), we can see that
\begin{equation}\label{vce}
   \langle v\rangle_\al =
\left\{
  \begin{array}{ll}
 \langle v^k \rangle_\al,  & \mbox{ if } \mbox{for some } k\in \mathbb{N}, \  \al > \al_k ,\\
 X, &  \mbox{ if } \al=0,
  \end{array}
\right.
\end{equation}
and hence
\begin{gather}\label{vceg}
 v \cap (X\times (\al_k, 1]) = v^{k} \cap (X\times (\al_k, 1]) \subseteq v^k.
   \end{gather}

We show that
$v\in C(X\times [0,1])$. To this end,
let $\{(x_l, \gamma_l)\}$ be a sequence in $v$
which converges to an element $(x,\gamma)$
in $X\times [0,1]$. If $\gamma=0$, then clearly $(x,\gamma) \in v$. If $\gamma>0$,
then there is a $k_0\in \mathbb{N}$ such that
$\gamma > \al_{k_0}$. Hence there is an $L$ such that $\gamma_l > \al_{k_0}$ when $l\geq L$.
So by \eqref{vceg},
 $(x_l, \gamma_l) \in v^{k_0}$ when $l\geq L$.
Since
$v^{k_0} \in K(X\times [\al_{k_0}, 1])$, it follows that $(x,\gamma) \in v^{k_0}\subset v$.

We claim that
\begin{description}
  \item[(c)] $\lim_{n\to\infty} H({\rm end}\, u_n^{(n)},  v) =0$ and $v\in F_{USCG} (X)^e$.
\end{description}

Let $n\in \mathbb{N}$ and $k\in \mathbb{N}$.
Then
by \eqref{fve},
\begin{align}\label{hle}
  H^* ({\rm end}\, u_n^{(n)},  v) &= \max\{  H^* ({\rm end}_{\al_k} \, u_n^{(n)},  v) ,\
 H^* ({\rm end}_0^{\al_k} \, u_n^{(n)},  v)    \} \nonumber \\
&
\leq \max\{H^*({\rm end}_{\al_k}\, u_n^{(n)},\,  v^{k}),\ \al_k\},
\end{align}
and by \eqref{vceg},
\begin{align}\label{hre}
H^* ( v, {\rm end}\, u_n^{(n)}) &  = \max\{
\sup_{(x, \gamma) \in  v \cap (X\times (\al_k, 1]} \overline{d} ( (x, \gamma),  \, {\rm end}\, u_n^{(n)} ),
\ H^*( v \cap (X\times [0, \al_k]),\, {\rm end}\, u_n^{(n)})\} \nonumber
\\
& \leq \max\{H^* ( v^{k},\, {\rm end}_{\al_{k}}\, u_n^{(n)}),    \ \al_k\}.
\end{align}
Clearly,
\eqref{hle} and \eqref{hre} imply that
\begin{align}\label{hlre}
  H ({\rm end}\, u_n^{(n)}, \, v) & \leq     \max\{  H ({\rm end}_{\al_k} \, u_n^{(n)}, \,  v^{k}),\ \al_k\}.
\end{align}

Now we show that
\begin{equation}\label{lgce}
\lim_{n\to\infty} H ({\rm end}\, u_n^{(n)},  v) =0.
\end{equation}
To see this,
let $\varepsilon>0$.
Notice that $\al_k \to 0$ and for each $\al_k$, $k=1,2,\ldots$,
$\lim_{n\to\infty} H({\rm end}_{\al_k}\, u_n^{(n)},   v^{k}) = 0$.
Then there is an $\al_{k_0}$ and an $N\in \mathbb{N}$ such that
 $\al_{k_0} < \varepsilon$
and
$ H({\rm end}_{\al_{k_0}}\, u_n^{(n)},   v^{k_0}) <\varepsilon$ for all $n\geq N$.
Thus by \eqref{hlre},
$H ({\rm end}\, u_n^{(n)},  v)<\varepsilon$ for all $n\geq N$. So \eqref{lgce} is true.

Since the sequence $\{ {\rm end}\, u_n^{(n)} \}$ is in $F_{USC}(X)^e$
and
$\{ {\rm end}\, u_n^{(n)} \}$ converges to $v$ in $(C(X \times [0,1]), H)$,
by
Proposition \ref{usce}, it follows that $v\in F_{USC}(X)^e$.

Let
 $k\in \mathbb{N}$.
Then
$v^{k}\in K(X\times [\al_k, 1])$,
and hence $\langle v^k \rangle_\al \in K(X) \cup \{\emptyset\}$ for all $\al\in [0,1]$.
So
from \eqref{vce}, $\langle v\rangle_\al \in K(X) \cup \{\emptyset\}$ for all $\al\in (0,1]$, and
thus $v\in F_{USCG}(X)^e$.

From
 affirmation (c), we have
that
 $\{u_n^{(n)}\}$ is a convergent sequence in $(F_{USCG} (X), H_{\rm end})$.
Note that $\{u_n^{(n)}\}$ is a subsequence of $\{u_n\}$.
Thus the proof is completed.

\end{proof}

\begin{tm} \label{tbfegnu}
  Let $U$ be a subset of $F_{USCG} (X)$. Then $U$ is totally bounded in $(F_{USCG} (X), H_{\rm end})$
if and only if
$U(\al)$
is totally bounded in $(X,d)$ for each $\al \in (0,1]$.
\end{tm}

\begin{proof}
 \textbf{\emph{Necessity}}.
The necessity part is Lemma \ref{tbfegnum}.

\textbf{\emph{Sufficiency}}.
Suppose that
$U(\al)$ is totally bounded
in $(X, d)$ for each $\al \in (0,1]$.
Then
$U(\al)$ is relatively compact
in $(\widetilde{X}, \widetilde{d})$ for each $\al \in (0,1]$.
Thus by Theorem \ref{rcfegnum},
$U$ is
 relatively compact in
 $(F_{USCG} (\widetilde{X}), H_{\rm end})$.
Hence
$U$ is
 totally bounded in
 $(F_{USCG} (\widetilde{X}), H_{\rm end})$.
So clearly
$U$
is
 totally bounded in
 $(F_{USCG} (X), H_{\rm end})$.

\end{proof}

\begin{tm}\label{cfegum}
  Let $U$ be a subset of $F_{USCG} (X)$. Then the following are equivalent:
\begin{enumerate}
\renewcommand{\labelenumi}{(\roman{enumi})}

\item
 $U$ is compact in $(F_{USCG} (X), H_{\rm end})$;

\item  $U(\al)$
is relatively compact in $(X, d)$ for each $\al \in (0,1]$ and $U$ is closed in $(F_{USCG} (X), H_{\rm end})$;

\item  $U(\al)$
is compact in $(X, d)$ for each $\al \in (0,1]$ and $U$ is closed in $(F_{USCG} (X), H_{\rm end})$.
\end{enumerate}

\end{tm}

\begin{proof}
By Theorem \ref{rcfegnum},
  (\romannumeral1) $\Leftrightarrow$ (\romannumeral2).
  Obviously (\romannumeral3) $\Rightarrow$ (\romannumeral2).
We
shall complete the proof by showing that
 (\romannumeral2) $\Rightarrow$
 (\romannumeral3).
 Suppose that (\romannumeral2) is true.
To verify (\romannumeral3),
it suffices to
show that $U(\al)$ is closed in $(X,d)$ for each $\al\in (0,1]$.
To do this,
let $\al\in (0,1]$
and
let $\{x_n\}$ be a sequence in $U(\al)$ with $\{x_n\}$ converges to an element $x$ in $(X,d)$.
We only need to
show that $x\in U(\al)$.

Pick
a sequence $\{u_n\}$ in $U$
such that
$x_n \in [u_n]_\al$ for $n=1,2,\ldots$, which means
 that $(x_n,\al) \in {\rm end}\, u_n$
for $n=1,2,\ldots$.

From the equivalence of (\romannumeral1) and (\romannumeral2),
 $U$ is compact in $(F_{USCG} (X), H_{\rm end})$.
So
there exists a subsequence $\{u_{n_k}\}$ of $\{u_n\}$ and $u\in U$
such that $H_{\rm end}(u_{n_k}, u) \to 0$.
Hence
by Remark \ref{hmr}, $\lim_{n\to \infty}^{(\Gamma)}  u_{n_k} = u $.
Note that $(x,\al) = \lim_{k\to\infty}(x_{n_k}, \al)$.
Thus
$$(x,\al) \in \liminf_{n\to \infty}{\rm end}\, u_{n_k}
=
{\rm end}\, u.$$
So
$x\in [u]_\al$, and therefore $x\in U(\al)$.

It can be seen that
Theorem 7.11 in \cite{huang719} can be verified in a similar manner to that in this theorem.

\end{proof}

\begin{tm}\label{cmu}
$ (X,d)$ is complete if and only if $F_{USCG} (X)$ is a closed set in $(F_{USC} (X), H_{\rm end})$.
\end{tm}

\begin{proof}
\textbf{\emph{Necessity}}. Suppose that $(X,d)$ is complete. To show that $F_{USCG} (X)$ is a closed set in $(F_{USC} (X), H_{\rm end})$,
  let $\{u_n\}$ be a sequence in $F_{USCG} (X)$ which converges to an element $u$ in $(F_{USC} (X), H_{\rm end})$.
It suffices to show that
 $u \in F_{USCG} (X)$.

For $\al \in (0,1]$, set $U(\al) : = \bigcup_{n=1}^{+\infty} [u_n]_\al$.
As the Cauchy sequence $\{u_n\}$ is obviously a totally bounded set in $(F_{USCG} (X), H_{\rm end})$,
by Lemma \ref{tbfegnum}, $U(\al)$
is totally bounded in $(X,d)$ for each $\al \in (0,1]$.
Since $X$ is complete, it follows that
 $U(\al)$ is relatively compact in $X$ for each $\al \in (0,1]$.
Then
by Theorem \ref{rcfegnum},
 $\{u_n\}$ is a relatively compact set in $(F_{USCG} (X), H_{\rm end})$.
Hence
 $\{u_n\}$ has a subsequence $\{u_{n_k}\}$
which converges to $v\in F_{USCG} (X)$ according to $H_{\rm end}$ metric.
Then
$u=v$ and thus $u\in  F_{USCG} (X)$.

\textbf{\emph{Sufficiency}}.
Suppose that $F_{USCG} (X)$ is a closed set in $(F_{USC} (X), H_{\rm end})$.
To show that
$ (X,d)$ is complete, we proceed by contradiction.
If this were not the case, then there is a Cauchy sequence $\{x_n\}$ in $X$ which is not a convergent sequence in $X$.
Set $S= \bigcup_{k=1}^{+\infty} \{x_k\}$. For $n=1,2,\ldots$, set $S_n = \bigcup_{k=1}^n \{x_k\}  $.
Then $\{ {S_n}_{F(X)} \}$ is a sequence in $F_{USCG} (X)$ and
$$H_{\rm end} ({S_n}_{F(X)}, S_{F(X)}) =\min\{   H({S_n}, S),\ 1 \} \to 0    \mbox{ as }    n\to \infty.$$
Clearly $S$ is not a compact set in $X$, this means that $S_{F(X)} \notin F_{USCG} (X)$.
So $F_{USCG} (X)$ is not closed in
 $(F_{USC} (X), H_{\rm end})$.

\end{proof}

\begin{pp}\label{spcm}

Let $\{x_n\}$ be a sequence in $X$. If $\{\widehat{x_n}\}$ converges to a fuzzy set $u$ in $F_{USC} (X)$
according to the $H_{\rm end}$ metric,
then
there is an $x\in X$ such that $u=\widehat{x}$ and $d(x_n, x)\to 0$ as $n\to\infty$.

\end{pp}

\begin{proof}

From Proposition \ref{sem}, $u\in F^{'1}_{USC}(X)$. So $[u]_\al\not=\emptyset$ for all
$\al\in [0,1)$.
We claim that
\\
(\romannumeral1) \ $[u]_\al$ is a singleton for all
$\al\in (0,1)$, and
\\
(\romannumeral2) \ there is an $x$ in $X$ such that $u=\widehat{x}$.

Assume that (\romannumeral1) is not true.
Then there is an $\al_1\in (0,1)$ such that $[u]_{\al_1}$ has at least two distinct elements.
Pick $p,q$ in $[u]_{\al_1}$ with $p \not= q$.

Let $x\in X$. Since $d(p,x) + d(q,x) \geq d(p,q)$, it follows that
$\max\{d(p,x), d(q,x)\} \geq \frac{1}{2} d(p,q)$.
Hence
\begin{align}\label{rsgn}
H_{\rm end}(u, \widehat{x})& \geq H^*({\rm end}\,u, {\rm end}\,\widehat{x}) \nonumber\\
 & \geq \min\{H^* ([u]_{\al_1}, \{x\}), {\al_1}\} \nonumber\\
 & \geq \min\{\max\{d(p,x), d(q,x)\}, {\al_1}\} \nonumber\\
 & \geq \min\{\frac{1}{2} d(p,q), {\al_1}\}.
\end{align}
From \eqref{rsgn} we have $H_{\rm end}(u, \widehat{x_n}) \geq \min\{\frac{1}{2} d(p,q), {\al_1}\}$ for all $n\in \mathbb{N}$.
This contradicts the fact that
$H_{\rm end}(u, \widehat{x_n}) \to 0$ as $n\to\infty$.
So $[u]_\al$ is a singleton for all
$\al\in (0,1)$; that is, (\romannumeral1) is true.

To show (\romannumeral2), we first show that $[u]_\al = [u]_\beta$ for all $\al, \beta$ in $(0,1)$.
Otherwise there is a pair $\al_0, \beta_0$ such that
 $0<\al_0<\beta_0<1$ and $[u]_{\al_0} \supsetneqq [u]_{\beta_0}$. This contradicts
 the fact that both $[u]_{\al_0}$ and $[u]_{\beta_0}$ are singletons.

Thus there is an $x$ in $X$ such that $[u]_\al = \{x\}$ for all
$\al\in (0,1)$.
Hence $[u]_0=[u]_1 = \{x\}$.
This means that $u=\widehat{x}$. So (\romannumeral2) is true.

From (\romannumeral2), $H_{\rm end} (\widehat{x_n}, u) = H_{\rm end} (\widehat{x_n}, \widehat{x}) =\min \{d(x_n, x), 1\}$.
Since $H_{\rm end} (\widehat{x_n}, u) \to 0$ as $n\to\infty$,
we have that $d(x_n, x) \to 0$ as $n\to\infty$.

\end{proof}

\begin{tm} \label{cum} The following statements are equivalent:
\\
  (\romannumeral1) \ $(X,d)$ is complete;
\\
   (\romannumeral2) \ $(F_{USCG} (X), H_{\rm end})$ is complete;
\\
  (\romannumeral3) \ $(F_{USC} (X), H_{\rm end})$ is complete.
\end{tm}

\begin{proof}

  (\romannumeral3)$\Rightarrow$(\romannumeral1).
 Suppose that $(F_{USC} (X), H_{\rm end})$ is complete. To show that $(X,d)$ is complete,
 let
 $\{x_n\}$ be a Cauchy sequence in $(X,d)$, we only need to show that
there is an $x$ in $X$ such that $d(x_n, x) \to 0$.

Note that $H_{\rm end} (\widehat{x},\,  \widehat{y}) = \min\{d(x,y),1\}$ for $x,y \in X$.
Then
$\{\widehat{x_n}\}$ is a Cauchy sequence in $(F_{USC} (X), H_{\rm end})$,
and
therefore $\{\widehat{x_n}\}$ converges to a fuzzy set $u$ in $F_{USC} (X)$ according to the $H_{\rm end}$ metric.
Thus by Proposition \ref{spcm}, there exists an $x\in X$
such
that $d(x_n,x)\to 0$.
So $(X,d)$ is complete.

  (\romannumeral1)$\Rightarrow$(\romannumeral3). Suppose that $(X,d)$ is complete.
  Then
   $(X\times[0,1], \overline{d})$ is complete, and
   therefore $(C(X\times[0,1]), H)$ is complete.
By Proposition \ref{usce}, $(F_{USC} (X), H_{\rm end})$ is complete.

    (\romannumeral3)$\Rightarrow$(\romannumeral2).
  Suppose that $(F_{USC} (X), H_{\rm end})$ is complete. Then, as
    (\romannumeral3)$\Rightarrow$(\romannumeral1) is true,  $(X,d)$ is complete.
  So by Theorem \ref{cmu},
   $(F_{USCG} (X)$ is closed in
  $(F_{USC} (X), H_{\rm end})$. Thus      $(F_{USCG} (X), H_{\rm end})$ is complete.

  (\romannumeral2)$\Rightarrow$(\romannumeral1). The proof is similar to that of (\romannumeral3)$\Rightarrow$(\romannumeral1).

  Suppose that $(F_{USCG} (X), H_{\rm end})$ is complete.
Let
 $\{x_n\}$ be a Cauchy sequence in $(X,d)$.
Then
$\{\widehat{x_n}\}$ is a Cauchy sequence in $(F_{USCG} (X), H_{\rm end})$,
and
therefore $\{\widehat{x_n}\}$ converges to a fuzzy $u$ in $F_{USCG} (X)$ according to the $H_{\rm end}$ metric.
Thus by Proposition \ref{spcm}, there exists an $x\in X$
such
that $d(x_n,x)\to 0$.
So $(X,d)$ is complete.

  Since we have shown that  (\romannumeral1)$\Rightarrow$(\romannumeral3),
   (\romannumeral3)$\Rightarrow$(\romannumeral2),
   and
     (\romannumeral2)$\Rightarrow$(\romannumeral1). The proof is completed.

\end{proof}

\section{Conclusions}

We discuss the compatibility of
 the endograph metric and the $\Gamma$-convergence
 on
  fuzzy sets in $\mathbb{R}^m$. The results in this paper significantly improve the corresponding results in
\cite{huang}.

It is known that the Fell topology is compatible with the Kuratowski convergence on $ C(\mathbb{R}^m)\cup \{\emptyset\}$.
So
the results in this paper indicate that Fell topology can be metrizable by the endograph metric
on a large class
 of fuzzy sets in $\mathbb{R}^m$.

The results in this paper have potential applications in the research
of
fuzzy sets involved the endograph metric and the $\Gamma$-convergence.

\end{document}